\documentclass{article}
\usepackage{CJK,CJKnumb,CJKulem,times,dsfont,ifthen,mathrsfs,latexsym,amsfonts,color}
\usepackage{amsmath,amsthm,makeidx,fontenc,amssymb,bm,graphicx,psfrag,listings,curves,extarrows,enumitem}
\usepackage[
linktocpage=true,colorlinks,citecolor=magenta,linkcolor=blue,urlcolor=magenta]{hyperref}

\usepackage{cite}
\usepackage{geometry}
\usepackage{indentfirst}

\usepackage{marvosym}

\usepackage{amssymb}
\makeatletter

\newcommand{\Rmnum}[1]{\expandafter\@slowromancap\romannumeral #1@}
\makeatother

\numberwithin{Assumption}{section} \numberwithin{Corollary}{section}
\numberwithin{Definition}{section} \numberwithin{equation}{section}
\numberwithin{Example}{section} \numberwithin{Lemma}{section}
\numberwithin{Proposition}{section} \numberwithin{Remark}{section}
\numberwithin{Theorem}{section}

\newtheorem{definition}{Definition}[section]
\newtheorem{theorem}{Theorem}[section]
\newtheorem{lemma}{Lemma}[section]
\newtheorem{corollary}{Corollary}[section]
\newtheorem{proposition}{Proposition}[section]
\newtheorem{remark}{Remark}[section]

\newcommand{\bt}{\begin{theorem}}
	\newcommand{\et}{\end{theorem}}
\newcommand{\bl}{\begin{lemma}}
	\newcommand{\el}{\end{lemma}}
\newcommand{\bd}{\begin{definition}}
	\newcommand{\ed}{\end{definition}}
\newcommand{\bc}{\begin{corollary}}
	\newcommand{\ec}{\end{corollary}}
\newcommand{\bp}{\begin{proof}}
	\newcommand{\ep}{\end{proof}}
\newcommand{\bx}{\begin{example}}
	\newcommand{\ex}{\end{example}}
\newcommand{\bi}{\begin{exercise}}
	\newcommand{\ei}{\end{exercise}}
\newcommand{\br}{\begin{remark}}
	\newcommand{\er}{\end{remark}}
\newcommand{\be}{\begin{equation}}
	\newcommand{\ee}{\end{equation}}
\newcommand{\bal}{\begin{align}}
	\newcommand{\bn}{\begin{enumerate}}
		\newcommand{\en}{\end{enumerate}}
	\newcommand{\ba}{\begin{align}}
		\newcommand{\ea}{\begin{align}}

			\newcommand{\bg}{\begin{align*}}
				\newcommand{\eg}{\end{align*}}
			\newcommand{\bcs}{\begin{cases}}
				\newcommand{\ecs}{\end{cases}}

			\newcommand{\N}{{\mathbb N}}

			\newcommand{\R}{{\mathbb R}}
			\newcommand{\bean}{\begin{eqnarray*}}
				\newcommand{\eean}{\end{eqnarray*}}

			\def\text#1{{\rm #1}}

			\def\dis{\displaystyle}

			\setitemize{itemindent=38pt,leftmargin=0pt,itemsep=-0.4ex,listparindent=26pt,partopsep=0pt,parsep=0.5ex,topsep=-0.25ex}

			\numberwithin{equation}{section}
			
\begin{document}
				\theoremstyle{plain}

				\title{\bf   Normalized solutions of mass supercritical Schr\"{o}dinger-Poisson equation with potential \thanks{ 	
						E-mails: pxq52918@163.com (X. Q. Peng),   mrizzi1988@gmail.com (M. Rizzi)}}
				
				\date{}
				\author{
					{\bf    Xueqin Peng$^{\mathrm{a},}$ \textsuperscript{\Letter},\; Matteo Rizzi$^{\mathrm{b}}$}\\
					\footnotesize \it  $^{\mathrm{a}}$ Department of Mathematical Sciences, Tsinghua University, Beijing 100084, China\\
				\footnotesize \it	$^{\mathrm{b}}$ Mathematisches Institut, Justus-Liebig-University Giessen, Giessen 35392, Germany
				}

				\maketitle
				\begin{center}
					\begin{minipage}{130mm}
						\begin{center}{\bf Abstract }\end{center}
						
					In this paper we prove the existence of normalized solutions $(\lambda,u)\subset (0,\infty)\times H^1(\mathbb{R}^3)$ to the following Schr\"{o}dinger-Poisson equation
					\begin{equation*}
						\begin{cases}
							-\Delta u+V(x)u+\lambda u+(|x|^{-1}\ast u^2)u=|u|^{p-2}u\quad\text{in}\quad\mathbb{R}^{3},\\
						\dis	u>0,\quad \int_{\mathbb{R}^{3}}u^2dx=a^2,
						\end{cases}
					\end{equation*}
				where $a>0$ is fixed, $p\in(\frac{10}{3},6)$ is a given exponent and the potential $V$ satisfies some suitable conditions. Since the $L^2(\R^3)$-norm of $u$ is fixed, $\lambda$ appears as a Lagrange multiplier. For $V(x)\geq0$, our solutions are obtained by using a mountain-pass argument on bounded domains and a limit process introduced by Bartsch et al \cite{bm}. For $V(x)\leq0$, we directly construct an entire mountain-pass solution with positive energy. 

							%
							%


							
							\vskip0.23in
							
						\end{minipage}
					\end{center}

					\vskip0.23in
					\section{Introduction}
					\setcounter{equation}{0}
					\setcounter{Assumption}{0} \setcounter{Theorem}{0}
					\setcounter{Proposition}{0} \setcounter{Corollary}{0}
					\setcounter{Lemma}{0}\setcounter{Remark}{0}
					\par
		Consider the following time-dependent Schr\"{o}dinger-Poisson equation
		\begin{equation}\label{eq1.1}
			i\partial_t\psi+\Delta\psi-(|x|^{-1}\ast |\psi|^2)\psi+g(|\psi|)\psi=0\qquad\text{in}~(0,\infty)\times\mathbb{R}^{3},
		\end{equation}
		where the function $\psi=\psi(t,x):(0,\infty)\times \mathbb{R}^3\rightarrow\mathbb{C}$ is complex valued and $g$ is real valued. This class of Schr\"{o}dinger-type equations with a repulsive nonlocal Coulomb potential are obtained as an approximation of the Hartree-Fock equation which is used to describe a quantum mechanical system of many particles, see \cite{L1977,L1987,M2001} for more physical background.
		\vskip 0.1in
		A stationary wave of (\ref{eq1.1}) is a solution of the form $\psi(t,x)=e^{i\lambda t}u(x)$, where $\lambda\in \mathbb{R}$ denotes the frequency. It is well-known that $e^{i\lambda t}u(x)$ solves (\ref{eq1.1}) if and only if $(\lambda,u)$ is a pair of solutions of the following equation
		\begin{equation}\label{eq1.2}
			-\Delta u+\lambda u+(|x|^{-1}\ast u^2)u=g(|u|)u\qquad\text{in}\,\R^3,
		\end{equation}
        We note that $\phi_u:=|x|^{-1}\ast u^2$ is the convolution of $u^2$ with the Green function of $-\Delta$ in $\R^3$, hence it fulfills $$-\Delta \phi_u=u^2\qquad\text{in}\,\R^3.$$
       
        If $\lambda\in \mathbb{R}$ is fixed in (\ref{eq1.2}), this kind of problem is known as the \textit{fixed frequency problem}, which has received many
		scholars' attention, and a great part of the literature is focused on existence, nonexistence and multiplicity of solutions to (\ref{eq1.2}) or similar problems, see \cite{He2020,wang1,PZAMP,pj,R2006} and the references therein. In those papers the authors apply variational methods. Similar methods to find solutions are applied in \cite{Tru, Rab1, SM, m1}.
		\vskip 0.1in
	On the other hand, it is also interesting to study the \textit{fixed mass problem}, that is looking for solutions with prescribed $L^2(\R^3)$-norm $\|u\|_2=a>0$. These kind of solutions, which have recently gained interest in physics, are known as \textit{normalized solutions}. Indeed, from the physical point of view, this approach turns out to be more meaningful since it offers a better insight of the variational properties of the stationary solutions of (\ref{eq1.2}), such as stability or instability. Compared to the fixed frequency case, the study of normalized solutions is more complicated because the terms in the corresponding energy functional scale differently, readers are invited to see \cite{BJ2013,BS2011,B2011,JL2013} for more details. Now we recall some known results in this direction. In \cite{B2011}, Bellazzini and Siciliano considered the following problem
		\begin{equation}\label{eq1.7}
			\begin{cases}
				-\Delta u+\lambda u+(|x|^{-1}\ast u^2)u=|u|^{p-2}u&\text{in}~\mathbb{R}^{3},\\
			\dis u>0,\int_{\mathbb{R}^3}u^2dx=a^2,
			\end{cases}
		\end{equation}
		for $p\in(2,3)$ and proved the existence of normalized solutions for $a>0$ small. After that, in \cite{BS2011}, they treated the case $p\in(3,\frac{10}{3})$ and proved that (\ref{eq1.7}) admits normalized solutions if $a>0$ is large enough. Later, in \cite{JL2013}, Jeanjean and Luo identified a threshold value of $a>0$ which separates existence and nonexistence of normalized solutions of (\ref{eq1.7}). It is noteworthy to point out that, in the \textit{$L^2$-subcritical case} $2<p<\frac{10}{3}$, one can obtain normalized solutions by considering the minimization problem
		\begin{equation*}
			\alpha(a):=\inf_{u\in \mathcal{S}_a}I(u),
		\end{equation*}
		where
		$I(u)$ is the energy functional corresponding to (\ref{eq1.7}) defined by
		\begin{equation*}
			I(u)=\frac{1}{2}\int_{\mathbb{R}^3}|\nabla u|^2dx+\frac{1}{4}B(u)
			-\frac{1}{p}\int_{\mathbb{R}^3}|u|^pdx\qquad\forall\,u\in H^1(\R^3),
		\end{equation*}
being
\begin{equation}
    B(u):=\int_{\mathbb{R}^3}\int_{\mathbb{R}^3}\frac{|u(x)|^2|u(y)|^2}{|x-y|}dxdy\qquad\forall\,u\in H^1(\R^3),
\end{equation}        
		and $\mathcal{S}_a$ is the $L^2$-constraint
		\begin{equation}\label{eq1.6}
			\mathcal{S}_a=\{u\in H^1(\mathbb{R}^3):~\|u\|_2^2=a^2\}.
		\end{equation}
		However, in the \textit{$L^2$-supercritical case} $\frac{10}{3}<p\leq 6$, the functional $I$ is no more bounded from below on $\mathcal{S}_a$ and the minimization method fails. For quite a long time, \cite{BJ2013} was the only paper dealing with the existence of normalized solutions for the $L^2$-supercritical case. The main difficulties and challenges are based on several aspects: (i)  the classical mountain-pass theorem is not applicable to construct a  $(PS)$ sequence; (ii) since $\lambda$ is unknown, the Nehari manifold cannot be used anymore, which brings more obstacles to prove the boundedness of the $(PS)$ sequence; (iii)  the compactness of the $(PS)$ sequence is more challenging even if one considers the radial case since the embedding of $H^1_r(\R^3)\hookrightarrow L^2(\R^3)$ is not compact.  
		\vskip0.1in
		Since we are interested in the \textit{$L^2$-supercritical case}, we now give more related results that inspire us to study our problem. In \cite{BJ2013}, Bellazzini, Jeanjean and Luo studied the normalized solutions by introducing a Nehari-Pohozaev manifold. The benefit of working on the Nehari-Pohozaev manifold is that $I(u)$ is coercive. They proved the existence of solutions to (\ref{eq1.7}) for $a>0$ sufficiently small. The smallness condition about $a$ is crucial and  it is used to show that the Lagrange multiplier $\lambda$ is positive which leads to compactness. Later, Luo \cite{L} studied  the multiplicity of normalized solutions by using the Fountain Theorem established by Bartsch and De Valeriola \cite{B2013}.  Chen \cite{CST} et al. generalized the results to more general non-linearities. Very recently, Jeanjean and Le in \cite{JL2021} investigated the following Schr\"{o}dinger-Poisson-Slater equation
		\begin{equation}\label{eq1.4}
			-\Delta u+\lambda u-\gamma (|x|^{-1}\ast |u|^2)u-b|u|^{p-2}u=0\qquad\text{in}~\mathbb{R}^3,
		\end{equation}
		with $p\in(\frac{10}{3},6]$, $\gamma,b\in\mathbb{R}$ and $\|u\|_2=a$ for some given $a>0$. They obtained existence and nonexistence results in the cases $(\gamma<0,b<0)$, $(\gamma>0,b>0)$ and $(\gamma>0,b<0)
$. In particular, for the case $\gamma>0,b>0$, they obtained two normalized solutions, one being a local minimizer and the other one being a mountain pass type solution for $a>0$ small.
		\vskip 0.1in
		The existence and multiplicity of normalized solutions to the autonomous Schr\"{o}dinger-Poisson equation are widely studied in the literature, but the non-autonomous case, that is $V(x)\neq0$, is less understood. From the perspective of physics, the presence of $V(x)$ is very important because it represents an external potential that influences the behavior of the stationary waves. From the mathematical point of view, the presence of $V$ brings more difficulties, especially in recovering compactness. In fact, $V$ is not required to be radial, hence the solutions are not expected to be radial, thus we cannot use the compactness of the embedding $H^1_{rad}(\R^3)\subset L^p(\R^3)$ for $2<p<6$. For this reason, we are led to consider the following problem
		\begin{equation}\label{p}
			\begin{cases}
				-\Delta u+V(x)u+\lambda u+(|x|^{-1}\ast|u|^2)u=|u|^{p-2}u&\text{in}~\mathbb{R}^{3},\tag{$\mathcal{P}$}\\
			\dis 	u>0,~ \int_{\mathbb{R}^{3}}u^2dx=a^2,
			\end{cases}
		\end{equation}
		where $a>0$, $V(x)$ is a fixed potential, $\frac{10}{3}<p<6$ and $\lambda\in \mathbb{R}$ is a Lagrange multiplier. The study of normalized solutions of (\ref{p}) is equivalent to looking for critical points of the functional
		\begin{equation}\label{eq1.5}
			J_{V}(u)=\frac{1}{2}\int_{\mathbb{R}^3}|\nabla u|^2dx+\frac{1}{2}\int_{\mathbb{R}^3}V(x)u^2dx+\frac{1}{4}B(u)
			-\frac{1}{p}\int_{\mathbb{R}^3}|u|^pdx
		\end{equation}
		on the constraint $\mathcal{S}_a$. Our purpose is to find suitable conditions on $V$ to prove existence of normalized solutions to \eqref{p}. Basically, we will distinguish two cases: $V(x)\geq0$ and $V(x)\leq0$. In both cases we will use a mountain-pass argument, with some relevant differences. In fact, in the case $V(x)\ge 0$, we will construct solutions in domains of the form $\Omega_r:=r\Omega$, where $\Omega\subset\R^3$ is a fixed bounded convex open set and then we will pass to the limit as $r\to\infty$ to prove existence of entire solutions in $\R^3$. On the other hand, if \( V(x) \leq 0 \), we directly construct entire solutions using the mountain-pass geometry of the functional $J_V$ and a classical Ghoussoub min-max principle described in \cite[Section 5]{G1993}, which enables us to construct a $(PS)$ sequence. 
		\vskip 0.1in
		Before formulating our main theorems, we first recall some known results. For the limit equation (\ref{eq1.7}), 		
		the existence of normalized solutions is studied in \cite{BJ2013}, at least for $a>0$ small enough. 
        We state the following result established by Bellazzini, Jeanjean and Luo.
		\begin{theorem}\label{th1}(\cite{BJ2013})
		Let $p\in(\frac{10}{3},6)$. Then there exists $a_0>0$ such that for any $a\in(0,a_0)$, there exists a solution $( u_a,\lambda_{a})\in H^1(\mathbb{R}^3)\times (0,\infty)$ of \eqref{eq1.7} with $I(u_a)=c_a$, where
			\begin{equation*}
				c_a=\inf_{g\in \mathcal{G}_a}\max_{t\in[0,1]}I(g(t))>0
			\end{equation*}
			\begin{equation*}
				\mathcal{G}_a=\{g\in C([0,1],\mathcal{S}_a):g(0)\in A_{K_a},I(g(1))<0\}
			\end{equation*}
			and $A_{K_a}=\{u\in \mathcal{S}_a:\|\nabla u\|_2^2\leq K_a\}$, for some constant $K_a>0$ small enough.
		\end{theorem}
		\begin{remark}
			In \cite{BJ2013}, in order to obtain
			a bounded $(PS)$ sequence, the authors introduced the auxiliary functional
			\begin{equation}\label{eq1.9}
				P(u):= \int_{\mathbb{R}^3}|\nabla u|^2dx+\frac{1}{4}B(u)
				-\frac{3(p-2)}{2p}\int_{\mathbb{R}^3}|u|^pdx.
			\end{equation}
			given by a linear combination of the Nehari and the Pohozaev constraints. This is relevant because, introducing the scaling
			\begin{equation}\label{scaling}
				u^t(x):=t^{\frac{3}{2}}u(tx)~\text{for}~t>0,\,u\in \mathcal{S}_a,
			\end{equation}
			we can see that $u^t\in \mathcal{S}_a$ for any $t>0$ so that, differentiating $I(u^t)$ with respect to $t$, we can prove that any solution $u$ to (\ref{eq1.7}) fulfills $P(u)=0$ (see also Lemma \ref{lemma-pohozaev} below). 		
		Moreover, in \cite{BJ2013} the authors also proved that, setting
			\begin{equation*}
				\mathcal{V}(a):=\{u\in \mathcal{S}_a:~P(u)=0\},
			\end{equation*}
		the solution $u_a$ constructed in Theorem \ref{th1}
satisfies			\begin{equation}\label{e2.2}
				c_a=I(u_a)=\color{black}\inf_{u\in \mathcal{V}(a)}I(u).
			\end{equation}
		\end{remark}
		
		
		As we mentioned above, the geometry of \( J_V \) is significantly influenced by the properties of \( V \). In the sequel, we will consider both non-negative and non-positive potentials separately. First, we impose the following conditions on $V(x)$:
		\vskip 0.1in
	\noindent	$(V_1)$ 
		$V\in C^1(\R^3),\,V\ge 0$, $\lim_{|x|\to\infty}V(x)=0$ and the function $W:x\mapsto \nabla V(x)\cdotp x\in L^{\infty}(\mathbb{R}^3)$. Moreover there exist $\theta \in(0,1),\eta\in (0,\infty)$ with $\eta+2\theta<\frac{6-p}{p-2}$ such that
		\begin{equation}\label{est-V-infty}
			\|V\|_\infty<\frac{2 \theta c_a}{a^2}\quad \text{and}\quad 
			\|W\|_\infty<\frac{c_a}{a^2}\eta.
		\end{equation}
		Alternatively, 
        we can assume that\\

      \noindent  $(V_1)'$ $V\in W^{3/2}_{loc}(\R^3),\,V\ge 0,\, \lim_{|x|\to\infty}V(x)=0$, there exists $q\in(3,\infty)$ such that $V,W\in L^q(\R^3)$.\\

        In addition to $(V_1)'$, we will need some smallness assumption about $\|V\|_q$ and $\|W\|_q$, which will be explained in Section $3$. Such a bound can be made explicit and depends on $p$ and $q$ only. However, for the reader's convenience we prefer to postpone the discussion of such a technical issue to Section $3$.\\
        
        Moreover, we also assume the following condition:
		\vskip 0.1in
\noindent $(V_2)$ There exist $\alpha,\,\rho\in(0,1)$ such that
		$$\limsup_{|y|\to\infty}\left(|y|^\alpha{{\rm ess} \sup}_{x\in B_{\delta|y|}(y)}\nabla V(y)\cdotp x\right) <0.$$
        
        \begin{remark}\label{remVne0}
        \begin{itemize}
        \item We note that assumption $(V_2)$ implies that $V\ne 0$. 
        \item We note that the essential $\sup$ reduces to the usual $\sup$ if $V\in C^1(B_{\delta|y|}(y))$. However, since we are interested in potentials that are not necessarily in $C^1(\R^3)$, we prefer to use the notion of ${\rm ess}\sup$.
        \end{itemize}
       
        \end{remark}
        
    Now we will exhibit examples of potentials fulfilling $(V_1)$-$(V_2)$ and examples of potentials fulfilling $(V_1)'$-$(V_2)$. We are also interested in non-radial potentials fulfilling these conditions, since in the radial case the proofs would be much easier, due to the compactness of the embedding $H^1_{rad}(\R^3)\subset L^p(\R^3)$, for $2<p<6$. However, our results do not require any radial symmetry, neither of the potential nor of the solutions.\\ 
    
    We note that condition $(V_2)$ is fulfilled if, for instance, $V(x)=(1+|x|)^{-\alpha}$ with $\alpha\in(0,1)$. In fact if we can see that
    $$|y|^\alpha \nabla V(y)\cdotp x=-\alpha|y|^\alpha(1+|y|)^{-\alpha-1}\frac{y}{|y|}\cdotp x\le-\alpha|y|^\alpha(1+|y|)^{-\alpha-1}\frac{y}{|y|}\cdotp((1+\rho)y)=-\alpha(1+\rho)\frac{|y|^{\alpha+1}}{(1+|y|)^{\alpha+1}},$$
    for any $\rho\in (0,1)$, $y\in\R^3\setminus B_1(0)$ and $x\in B_{\rho|y|}(y)$. Multiplying by a sufficiently small constant $c>0$, the potential $cV$ also fulfills $(V_1)$. Taking a non-constant $C^1$ function $\varphi:S^2\to(0,\infty)$ defined on the unit sphere with $\sup_{x\in S^2}|\nabla \varphi(x)|$ small enough, we get a non-radial potential $c\varphi\left(\frac{x}{|x|}\right) V(x)$ which satisfies $(V_1)-(V_2)$ provided $c>0$ is small enough.\\

    In order to show an example of potential $V$ satisfying $(V_1)'-(V_2)$, we take $q>3$, $\alpha\in\left(\frac{3}{q},1\right)$, $\beta\in\left(0,\frac{q}{3}\right)$ and set 
    $$V(x):=\begin{cases}
    c(1+|x|)^{-\alpha}\qquad\forall\,x\in B_1(0)\\
    2^{-\alpha}c|x|^{-\beta}\qquad\forall\,x\in\R^3\setminus B_1(0),
    \end{cases}$$
    where $c>0$ is a sufficiently small constant. We note that such a potential is unbounded in a neighbourhood of the origin. Multiplying it by $\varphi\left(\frac{x}{|x|}\right)$ as above and changing, if necessary, the value of $c$, we get an example of non-radial potential fulfilling $(V_1)'-(V_2)$.
    \vskip0.1in
    To summarize, in the case $V(x)\ge 0$ we have the following results.
        \begin{theorem}\label{theorem1.2} Assume that $\frac{10}{3}<p<6$ and $(V_1)-(V_2)$ hold. Then there exists a constant $a^*>0$ such that for any $a\in(0,a^*)$, (\ref{p}) possesses a solution $(\lambda,u)\in (0,\infty)\times H^1(\R^3)$ such that $J_V(u)\ge c_a>0$.
		\end{theorem}
        Similarly, in case $(V_1)'$ holds, we have the following result.
        \begin{theorem}\label{theorem1.3} Assume that $\frac{10}{3}<p<6$ and  $(V_1)'-(V_2)$ hold. Then there exists a constant $a^*>0$ and a constant $\kappa=\kappa(p,q)>0$ depending on $p$ and $q$ such that if $a\in(0,a^*)$ and $$\max\{\|V\|_q,\|W\|_q\}<\kappa,$$ 
        (\ref{p}) possesses a solution $(\lambda,u)\in (0,\infty)\times H^1(\R^3)$ such that $J_V(u)\ge c_a>0$.
		\end{theorem}       
        We note that assumption $(V_2)$ implies that $V\ne 0$, as we observed in Remark \ref{remVne0}. However, in case $V\equiv 0$, Theorems \ref{theorem1.2} and \ref{theorem1.3} reduce to Theorem \ref{th1} of \cite{BJ2013}, hence the result still holds true.\\ 
        
        Now we consider the case $V(x)\leq0$, with some assumptions about $V(x)$:
		\vskip 0.1in
		\noindent $(V_3)$ $V(x)\leq 0$, $V(x)\not\equiv0$ and $\lim\limits_{|x|\rightarrow\infty}V(x)=\sup\limits_{x\in\R^3}V(x)=0$.\\
		$(V_4)$ $V(x)\in L^{\frac{3}{2}}(\R^3)$, $\Tilde{W}:x\rightarrow V(x)|x|\in L^{3}(\R^3)$ satisfying $\|V\|_{\frac{3}{2}}<\frac{1}{2}S$ and
		$$3\left(\frac{2(p-2)^2}{6-p}+(p-4)^+\right)S^{-1}\|V\|_{\frac{3}{2}}
		+4\left(\frac{3(p-2)^2}{6-p}+1\right)S^{-\frac{1}{2}}\|\tilde{W}\|_3<3p-10,$$
		where $$S:=\inf_{u\in D^{1,2}(\R^3)\setminus\{0\}}\frac{\|\nabla u\|^2_2}{\|u\|^2_6}=3\pi\left(\frac{\Gamma(3/2)}{\Gamma(3)}\right)^{2/3}>0$$
        is the Aubin-Talenti constant (see \cite{T}) and $\Gamma$ is the $\Gamma$-function. Equivalently, $S^{-1/2}$ is the best constant in the embedding $D^{1,2}(\mathbb{R}^3)\hookrightarrow L^6(\mathbb{R}^3)$.
		\begin{theorem} \label{thm1.3}Assume that $\frac{10}{3}<p<6$ and $(V_3)-(V_4)$ hold. Then there exists $a_*>0$ small such that for any $a\in(0,a_*)$, (\ref{p}) possesses a mountain-pass solution $(\tilde{\lambda},\tilde{u})\in(0,\infty)\times H^1(\R^3)$ such that $J_V(\tilde{u})>0$. 
		\end{theorem}
		In the paper by Molle et al. \cite{MRV}, it was demonstrated that, if we consider the problem
        \begin{equation}\notag
        \begin{cases}
        -\Delta u+\lambda u+V(x)u=|u|^{p-2}u&\text{in}\,\R^N,\\
        \dis \dis
        u\ge 0,\,\int_{\R^N}u^2dx=a^2
        \end{cases}
        \end{equation}
        with $V\le 0$, in addition to the mountain-pass solution we also have a local minimizer with negative energy, at least under suitable assumptions about $V$. However, when attempting to construct a local minimizer with negative energy for \eqref{p} using the approach outlined in \cite{MRV}, the analysis of the concave-convex characteristics of the associated functional becomes significantly challenging due to the positivity of the non-local term. This issue remains unresolved, and we defer its exploration to future work.

		\color{black}\begin{remark}
			We recall the following two inequalities, which are essential for the main proof.
            \vskip 0.1in
			\begin{itemize}
				\item Gagliardo-Nirenberg inequality \cite{L1959}:  for any $N\geq 3$ and $p\in[2,2^*]$ we have
				\begin{equation}
					\|u\|_p\leq C(N,p)\|\nabla u\|_2^{\mu}\|u\|_2^{1-\mu},
				\end{equation}
				where $\mu=N(\frac{1}{2}-\frac{1}{p})$ and $2^*=\frac{2N}{N-2}$.
				\item Hardy-Littlewood-Sobolev inequality \cite{LL2001}: for $f\in L^p(\mathbb{R}^N)$, $g\in L^q(\mathbb{R}^N)$ and $0<s<N$,
				\begin{equation}
					\left|\int_{\mathbb{R}^N}\int_{\mathbb{R}^N}\frac{f(x)g(y)}{|x-y|^{s}}dxdy\right|\leq C(N,s,p,q)\|f\|_p\|g\|_q,
				\end{equation}
				where $p,q>1,\frac{1}{p}+\frac{1}{q}+\frac{s}{N}=2.$
			\end{itemize}
		\end{remark}

				\vskip0.1in

				The paper is organized as follows. Section \ref{secV_1} is dedicated to proving Theorem \ref{theorem1.2} under the conditions $(V_1)-(V_2)$, Section \ref{sec_thm1.3} is devoted to the proof of Theorem \ref{theorem1.3} if $(V_1)'-(V_2)$ hold, and in Section \ref{sec3}, we consider the non-positive potential case and  give the proof of Theorem \ref{thm1.3}.

				\medspace
				\vskip 0.15in
				 \textbf{Notations} 
				Throughout this paper, we make use of the following notations:
				\vskip 0.1in
				\begin{itemize}
					\item $L^p(\mathbb{R}^3)$ ($p\in[1,\infty))$) is the Lebesgue space equipped with the norm
					\begin{equation*}
						\|u\|_p=\left(\int_{\mathbb{R}^3}|u|^pdx\right)^{\frac{1}{p}};
					\end{equation*}
                     \item $L^{\infty}(\mathbb{R}^3)$  is the Lebesgue space equipped with the norm
                     \begin{equation*}
						\|u\|_{\infty}=\mathop{\text{ess~sup}}\limits_{x\in \mathbb{R}^3}|u(x)|;
					\end{equation*}
					\item $H^1(\mathbb{R}^3)$ denotes the usual Sobolev space endowed with the norm
					\begin{equation*}
						\|u\|=\left(\int_{\mathbb{R}^3}(|\nabla u|^2+|u|^2)dx\right)^{\frac{1}{2}};
					\end{equation*}
					\item $D^{1,2}(\mathbb{R}^3)$ is the Banach space given by
					\begin{equation*}
						D^{1,2}(\mathbb{R}^3)=\{u\in L^6(\mathbb{R}^3):\nabla u\in L^2(\mathbb{R}^3)\};
					\end{equation*}
					\item $\tilde{C}$, $c,\tilde{c}$ represent positive constants whose values may change from line to line;
					\item $``\rightarrow"$ and $``\rightharpoonup"$ denote the strong and weak convergence in the related function spaces respectively;
					\item $o(1)$ denotes the quantity that tends to $0$;
					\item For any $x\in\mathbb{R}^3$ and $r>0$, $B_r(x):=\{y\in\mathbb{R}^3:|y-x|<r\}$.
				\end{itemize}

				\vskip0.23in
			\section{The proof of Theorem \ref{theorem1.2}: differentiable non-negative potential}\label{secV_1}
	\setcounter{equation}{0}

	\setcounter{Assumption}{0} \setcounter{Theorem}{0}
	\setcounter{Proposition}{0} \setcounter{Corollary}{0}
	\setcounter{Lemma}{0}\setcounter{Remark}{0}
	\par

		Throughout this Section we will assume that $(V_1)$ holds. The strategy to prove Theorem \ref{theorem1.2} is the following: first we fix a bounded convex open set $\Omega\subset\R^3$ and solve the problem
			\begin{equation}
				\label{eq-Omega_r}
				\begin{cases}
					-\Delta u+V(x)u+\lambda u+(|x|^{-1}\ast u^2) u=|u|^{p-2}u\quad \text{in }\quad \Omega_r\\
				\dis	\dis u\in H^1_0(\Omega_r),\,\int_{\Omega_r}u^2dx =a^2,
				\end{cases}
			\end{equation}
                where $\Omega_r:=r\Omega$ for $r>0$ large enough, where $$\phi_u(x)=\int_{\R^3}\frac{u^2(y)}{|x-y|}dy,\qquad\forall\,x\in\R^3.$$
            is well defined by extending $u$ to $0$ in $\R^3\setminus\Omega_r$. In other words $\phi_u=|x|^{-1}\ast u$ is the convolution of $u$ with the Green function of the Laplacian in $\R^3$, hence it satisfies the equation
            $$-\Delta \phi_u=u^2\qquad\text{in}\,\R^3.$$
            Then we will prove that our solutions are bounded uniformly in $r$ in $H^1(\Omega_r)$, so that we can pass to the limit as $r\to\infty$ and prove existence of a normalized  solution to (\ref{p}) in the whole $\R^3$. We will see that the construction of solutions in a large bounded domain does not require condition $(V_2)$, which is only used to pass to the limit to have a normalized  solution in $\R^3$.
			
			\begin{theorem}\label{th-bd-V_1}
			Assume that $(V_1)$ holds.	Let $\Omega\subset\R^3$ be a convex bounded open set and $a_0>0$ be given in Theorem \ref{th1}.  Then for any $a\in(0,a_0)$, there exists $r_a>0$ such that, for any $r\geq r_a$, 
				Problem \eqref{eq-Omega_r} has a solution $(\lambda_{r},u_{r})\in\R\times H^1_0(\Omega_r)$ with $u_{r}\ge 0$ in $\Omega_r$.
			\end{theorem}
			Moreover, we will obtain a result about global boundedness of the solutions constructed in Theorem \ref{th-bd-V_1}.
			\begin{theorem}\label{th-unif-bound}
			 Assume that $(V_1)$ holds.	Let $\Omega\subset\R^3$ be a convex bounded open set. Then there exists $a^*\in(0,a_0)$ such that for any $a\in(0,a^*)$ and $r\geq r_a$, the solution $(\lambda_r,u_r)$  constructed in Theorem \ref{th-bd-V_1} satisfying 
				\begin{equation}\label{bound-lambda}
					0<\liminf_{r\to\infty} \lambda_r\le\limsup_{r\to\infty}\lambda_r<\infty
				\end{equation}
				and
				\begin{equation}\label{bound-u-r-infty}
					\sup_{r\geq r_a}\|u_r\|_\infty<\infty.
				\end{equation}
			\end{theorem}
			Section \ref{secV_1} is organised as follows. In Subsection \ref{subsec-th-bd-V_1} we prove Theorem \ref{th-bd-V_1} about normalized  solutions in large bounded domains, which is itself an interesting result, then in Subsection \ref{sec-lim-V_1} we prove Theorem \ref{th-unif-bound} (see Remarks \ref{rem-lambda-r>0} and \ref{rem-bound-u-infty}) and we pass to the limit as $r\to\infty$ in order to conclude the proof of Theorem \ref{theorem1.2} under assumptions $(V_1)$ and $(V_2)$.
			
			\color{black}\subsection{Proof of Theorem \ref{th-bd-V_1}}\label{subsec-th-bd-V_1}
			The proof of Theorem \ref{subsec-th-bd-V_1} requires several steps.			
			Up to a translation, we can assume without loss of generality that $0\in\Omega$, so that $\Omega_{r_1}\subset\Omega_{r_2}$ if $r_1<r_2$.\\
			
		For given $r>0$, set $$\mathcal{S}_{r,a}:=\{u\in H^1_0(\Omega_r):\,\|u\|_{L^2(\Omega_r)}=a\}.$$ 
		Taking $s\in[\frac{1}{2},1]$ and we aim to find a critical point $u\in \mathcal{S}_{r,a}$ of the energy functional
			$$E_{r,s}(u):=\frac{1}{2}\int_{\Omega_r}|\nabla u|^2 dx+\frac{1}{2}\int_{\Omega_r} V(x) u^2 dx+\frac{1}{4}B(u)-\frac{s}{p}\int_{\Omega_r}|u|^p dx\qquad\forall\, u\in H^1_0(\Omega_r)$$
			constrained to $\mathcal{S}_{r,a}$, for almost every $s\in[\frac{1}{2},1]$. We note that $B(u)$ is well defined for $u\in H^1_0(\Omega_r)$ by extending $u$ to $0$ outside $\Omega_r$ and, by the Hardy-Littlewood-Sobolev and the Gagliardo-Nirenberg inequalities, one has
			\begin{equation}\label{local-term}
				B(u)\leq \bar{C}\|\nabla u\|_2\|u\|_2^3.
			\end{equation}
            Such critical points are the solutions to the problem
			\begin{equation}
				\label{eq-Omegar-s}
				\begin{cases}
					-\Delta u+V(x)u+\lambda u+\phi_u u=s|u|^{p-2}u&\text{in}\,\Omega_r,\\
				\dis	u\in H^1_0(\Omega_r),\,\int_{\Omega_r}u^2 dx=a^2.
				\end{cases}
			\end{equation}
			
			As above, we extend $u$ to $0$ outside $\Omega_r$ when considering the convolution product $|x|^{-1}\ast |u|^2$. These critical points are found by a mountain-pass argument which relies on an argument developed in \cite[Theorem 1.5]{BCJS}. After that we will consider the limit as $s\to 1$ and obtain a solution to (\ref{eq-Omega_r}).\\

			For future purposes we introduce the notations
			\begin{equation}\notag
				\begin{aligned}
					E_{\infty,s}(u)&:=\frac{1}{2}\int_{\R^3}|\nabla u|^2 dx+\frac{1}{2}\int_{\R^3} V (x)u^2 dx+\frac{1}{4}B(u)-\frac{s}{p}\int_{\R^3}|u|^p dx,\qquad\forall\, u\in H^1(\R^3)\\
					\bar{E}_{r,s}(u)&:=\frac{1}{2}\int_{\Omega_r}|\nabla u|^2 dx+\frac{1}{4}B(u)-\frac{s}{p}\int_{\Omega_r}|u|^p dx,\qquad\forall\, u\in H^1_0(\Omega_r)\\
					\bar{E}_{\infty,s}(u)&:=\frac{1}{2}\int_{\R^3}|\nabla u|^2 dx+\frac{1}{4}B(u)-\frac{s}{p}\int_{\R^3}|u|^p dx,\qquad\forall\, u\in H^1(\R^3).
				\end{aligned}
			\end{equation}
		We note that, in these notations, $E_{\infty,1}=J_V$ and $\bar{E}_{\infty,1}=I$.\\	
        
        For functions $u\in H^1(\R^3)$
			and $t>0$, we recall the scaling $u^t(x):=t^\frac{3}{2}u(tx)$. 	We will show that the functional $E_{r,s}$ has the mountain-pass geometry.
			 \bl \label{lemma-MP-geom}
				Assume that $(V_1)$ holds. Let $a_0>0$ be given in Theorem \ref{th1}. Then for any $a\in(0,a_0)$, there exist $r_a>0,\,\tilde{c}_a>c_a$ and $u^0,\,u^1\in\mathcal{S}_{r_a,a}$ such that, setting
				$$\Gamma_{r,a}:=\{\gamma\in C([0,1],\mathcal{S}_{r,a}):\,\gamma(0)=u^0,\,\gamma(1)=u^1\},$$
				the mountain-pass level
				\begin{equation}\label{MP-level}
					m_{r,s}(a):=\inf_{\gamma\in\Gamma_{r,a}}\max_{t\in[0,1]} E_{r,s}(\gamma(t))
				\end{equation}
				fulfils
				\begin{equation}\label{bound-m-rs-unif-rs}
					\max\{E_{r,s}(u^0),E_{r,s}(u^1)\}<c_a\le m_{r,s}(a)\le\tilde{c}_a
					\qquad\forall\,s\in[\frac{1}{2},1],
					\qquad\forall\,r\geq r_a.
				\end{equation}
				and
				\begin{equation}
					\label{up-est-m_r1}
					m_{r,1}(a)<(1+\theta)c_a\qquad\forall\,r\geq r_a.
				\end{equation}
		\el
			\begin{proof}

				For any $a\in(0,a_0)$, $t>0$ and $r>0$, we consider the paths 
				$$\gamma_r:t\in(0,\infty)\mapsto\frac{\chi_r  (u_a)^t}{\|\chi_r  (u_a)^t\|_2}a\in \mathcal{S}_{r,a},\qquad\gamma_\infty:t\in(0,\infty)\mapsto (u_a)^t\in \mathcal{S}_a,$$
				where $u_a$ is the solution of (\ref{eq1.7}) given in Theorem \ref{th1} and $\chi_r:\R^3\to[0,1]$ is a smooth cutoff function such that $\chi_r=1$ in $\Omega_{r-1}$ and $\chi_r=0$ in $\R^3\setminus \Omega_{r}$.\\ 
				
				\begin{itemize}

					\item\textit{Construction of $u^1$}\\
					
					First we note that, for any $a\in(0,a_0)$ and $t>0$ we have
					\begin{equation}\label{main-est-E-rs}
						\lim_{r\to\infty}E_{r,\frac{1}{2}}(\gamma_r(t))=E_{\infty,\frac{1}{2}}(\gamma_\infty(t))\le \bar{E}_{\infty,\frac{1}{2}}(\gamma_\infty(t))+\frac{1}{2}\|V\|_\infty a^2<\bar{E}_{\infty,\frac{1}{2}}(\gamma_\infty(t))+\theta c_a.
					\end{equation}
					As a consequence, there exists $\bar{t}_a>0$ such that
					$$E_{r,s}(\gamma_r(t))\le E_{r,\frac{1}{2}}(\gamma_r(t))<\bar{E}_{\infty,\frac{1}{2}}(\gamma_\infty(t))+\theta c_a <0\qquad\forall\, s\in[\frac{1}{2},1],\,\forall\, r\ge r_a,\,\forall\, t\ge \bar{t}_a.$$
					if $r_a>0$ is large enough. In particular, taking $u^1:=\gamma_{r_a}(\bar{t}_a)\in\mathcal{S}_{r_a,a}$, we have
					\begin{equation}\label{E(u1)}
						E_{r,s}(u^1)=E_{r_a,s}(u^1)<0\qquad\forall\, s\in[\frac{1}{2},1],\,\forall\,r\ge r_a.
					\end{equation}
					\item\textit{Construction of $u^0$}\\
					
					In order to construct $u^0$ we note that, for any 
					$a\in(0,a_0)$, there exists $t_{1,a}>0$ such that, for any $0<t\le t_{1,a}$ we have
					$$\lim_{r\to\infty}\|\nabla\gamma_r(t)\|^2_2=\|\nabla\gamma_\infty(t)\|^2_2=t^2\|\nabla u_a\|^2_2<K_a.$$
					As a consequence, for any $a\in(0,a_0)$ and $0<t\le t_{1,a}$, we have
					$$\|\nabla\gamma_r(t)\|^2_2<K_a\qquad\forall\,r\ge r_a,$$
					if $r_a>0$ is large enough. Moreover, by (\ref{main-est-E-rs}), for any $\epsilon\in (0,1-\theta)$ there exists $t_{2,a}>0$ such that, for any $0<t\le t_{2,a}$, we have
					\begin{equation}
						\label{est-E-u0}
                   \lim_{r\to\infty}E_{r,\frac{1}{2}}(\gamma_r(t))<\bar{E}_{\infty,\frac{1}{2}}(\gamma_\infty(t))+\theta c_a<(\epsilon+\theta)c_a<c_a.
					\end{equation}
					As a consequence, if $r_a>0$ is large enough,
					$$E_{r,s}(\gamma_r(t))\le E_{r,\frac{1}{2}}(\gamma_r(t))<c_a\qquad\forall\, s\in[\frac{1}{2},1],\,\forall\,r\ge r_{a},\,\forall\,0<t\le t_{2,a}.$$
					Taking $\underline{t}_a:=\min\{t_{1,a},\,t_{2,a}\}>0$ and 
					setting $u^0:=\gamma_{r_a}(\underline{t}_a)$, we have
					\begin{equation}\label{E(u0)}
						\|\nabla u^0\|^2_2<K_a,\qquad E_{r,s}(u^0)=E_{r_a,s}(u^0)<c_a\qquad\forall\, s\in[\frac{1}{2},1],\,\forall\,r\ge r_a.
					\end{equation}
					To summarise  (\ref{E(u1)}) and (\ref{E(u0)}), we have $u^0,\,u^1\in\mathcal{S}_{r_a,a}$ and
					$$\max\{E_{r,s}(u^0),E_{r,s}(u^1)\}< c_a\qquad\forall\,r\ge r_a,\,\forall\, s\in[\frac{1}{2},1]$$
					if $r_a>0$ is large enough.
                    \vskip0.1in
					\item \textit{Lower estimate of the mountain pass level: for given  $0<a<a_0$, $m_{r,s}(a)\ge c_a$, for any $s\in[\frac{1}{2},1]$ and  $r\ge r_a$, if $r_a>0$ is large enough.}
					\vskip 0.1in
					It follows from (\ref{E(u1)}) and (\ref{E(u0)}), we have $\Gamma_{r,a}\subset \mathcal{G}_a$, for any $a\in(0,a_0)$ and $r\ge r_a$, which yields that, for any $\gamma\in \Gamma_{r,a}$,
					$$\max_{t\in[0,1]}E_{r,s}(\gamma(t))\ge \max_{t\in[0,1]}\bar{E}_{r,s}(\gamma(t))=\max_{t\in[0,1]}\bar{E}_{\infty,s}(\gamma(t))\ge \max_{t\in[0,1]}\bar{E}_{\infty,1}(\gamma(t))\ge c_a.$$
					Taking the infimum over $\Gamma_{r,a}$ we have the required estimate.
					\vskip 0.1in
					\item \textit{Upper estimate of the mountain-pass level: for given $0<a<a_0$ there exists $\tilde{c}_a>c_a$ such that $m_{r,s}(a)\le\tilde{c}_a$, for any $s\in[\frac{1}{2},1]$ and  $r\ge r_a$, if $r_a>0$ is large enough.}
						\vskip 0.1in
					In order to prove the upper estimate, it is sufficient to show that for any $a\in(0,a_0)$, there exists $\tilde{c}_a>c_a$ such that $$\max_{t>0}E_{r,s}(\gamma_r(t))\le\tilde{c}_a\qquad\forall\,r\ge r_a,\,\forall\, s\in [\frac{1}{2},1].$$
					First we note that
					$$E_{r,s}(\gamma_r(t))\le \bar{E}_{r,s}(\gamma_r(t))+\frac{1}{2}\|V\|_\infty a^2
					\qquad\forall\,r>0,\,\forall\, s\in [\frac{1}{2},1],\,\forall\,t>0.$$
					Letting $r\to\infty$, we have
					\begin{equation}\notag
						\lim_{r\to\infty}E_{r,s}(\gamma_r(t))\le\bar{E}_{\infty,s}(\gamma_\infty(t))+\frac{1}{2}\|V\|_\infty a^2<\bar{E}_{\infty,s}(\gamma_\infty(t))+\theta c_a\qquad\forall\, t>0,s\in[\frac{1}{2},1],
					\end{equation}
					so in particular
					$$\lim_{r\to\infty}E_{r,s}(\gamma_r(t))<\bar{c}_a(s)+\theta c_a\le \bar{c}_a(1/2)+\theta c_a\qquad\forall\, t>0,s\in[\frac{1}{2},1],$$
					where we have set $\bar{c}_a(s):=\max_{t>0}\bar{E}_{\infty,s}(\gamma_\infty(t))$. 
				Recalling that $\gamma_{r_a}(\underline{t}_a)=u^0$, $\gamma_{r_a}(\bar{t}_a)=u^1$,   by the lower estimate of $m_{r,s}(a)$, one has
					\begin{equation}\notag
						\begin{aligned}
							&c_a\le m_{r,s}(a)\le \max_{\tau\in[0,1]}E_{r,s}(\gamma_{r_a}((1-\tau)\underline{t}_a+\tau\bar{t}_a))\\
							&=\max_{t>0}E_{r,s}(\gamma_{r_a}(t))< \bar{c}_a(s)+\theta c_a\le \bar{c}_a(1/2)+\theta c_a=:\tilde{c}_a
						\end{aligned}
					\end{equation}
					for any $r\ge r_a,\,s\in[\frac{1}{2},1]$ if $r_a>0$ is large enough. In particular, if $s=1$, using that $\bar{c}_a(1)=c_a$, our argument shows that
					$$m_{r,1}(a)<\bar{c}_a(1/2)+\theta c_a=c_a(1+\theta)\qquad\forall\, r\geq r_a.$$
				\end{itemize}
			\end{proof}
			Applying \cite[Theorem 1.5]{BCJS}, we will show the existence of a solution $(\lambda_{r,s},u_{r,s})$ to (\ref{eq-Omegar-s}) with $u_{r,s}\ge 0$ for almost every $s\in[\frac{1}{2},1]$. 
			\begin{proposition}
				\label{prop-sol-Omegar-s}
				Assume that $(V_1)$ holds. Let $0<a<a_0$ and $r_a>0$ be given in Lemma \ref{lemma-MP-geom}. Then for almost every $s\in[\frac{1}{2},1]$ and $r\geq r_a$, Problem (\ref{eq-Omegar-s}) admits a solution $(\lambda_{r,s},u_{r,s})$ with $u_{r,s}\ge 0$ and $E_{r,s}(u_{r,s})=m_{r,s}(a)$.
			\end{proposition}
			\begin{proof}
				For given $a\in(0,a_0)$ and $r\geq r_a$, we shall apply  \cite[Theorem 1.5]{BCJS} to $E_{r,s}$ with $\Gamma_{r,a}$ defined in Lemma \ref{lemma-MP-geom}. Set
				$$A_1(u):=\frac{1}{2}\int_{\Omega_r}|\nabla u|^2 dx+\frac{1}{2}\int_{\Omega_r}V(x)u^2 dx+\frac{1}{4}B(u),\qquad A_2(u):=\frac{1}{p}\int_{\Omega_r}|u|^p dx,$$
			thus $E_{r,s}(u)=A_1(u)-sA_2(u)$. By Lemma \ref{lemma-MP-geom} and \cite[Theorem 1.5]{BCJS}, there exists a bounded $(PS)$ sequence $\{u_n\}$ of $E_{r,s}$ at level $m_{r,s}(a)$ constrained to $\mathcal{S}_{r,a}$, for almost every $s\in[\frac{1}{2},1]$. In other words, for almost every $s\in[\frac{1}{2},1]$, there exists a bounded sequence $\{u_n\}\subset H^1_0(\Omega_r)$ such that
				$$E_{r,s}(u_n)\to m_{r,s}(a),\qquad \nabla E_{r,s}(u_n)=\lambda_n u_n+o(1)\,\text{in}\, H^{-1}(\Omega_r),$$
				where $$a^2\lambda_n:=-\int_{\Omega_r}|\nabla u_n|^2 dx-\int_{\Omega_r}V(x)u_n^2 dx-B(u_n)+s\int_{\Omega_r}|u_n|^p dx+o(1)\|u_n\|_{H^1(\Omega_r)}.$$
				As a consequence, there exist $(\lambda_{r,s},u_{r,s})\in\R\times \mathcal{S}_{r,a}$ such that, up to a subsequence, 
				\begin{equation*}
				\lambda_n\to \lambda_{r,s}~\text{in}~\R, ~u_n\rightharpoonup u_{r,s} ~\text{weakly~ in}~H^1_0(\Omega_r),~u_n\rightarrow u_{r,s} ~\text{strongly ~ in}~L^q(\Omega_r)~\text{ for}~q\in[2,6).
				\end{equation*}
			Hence one can see that $(\lambda_{r,s},u_{r,s})$ is a solution to (\ref{eq-Omegar-s}). 
			 
				Note that
				\begin{equation}
					\nabla E_{r,s}(u_n)u_n=\lambda_n a^2+o(1),\qquad
					\nabla E_{r,s}(u_n)[u_{r,s}]=\lambda_n \int_{\Omega_r}u_nu_{r,s}dx+o(1)
				\end{equation}
				and
				\begin{equation}
					\int_{\Omega_r}V(x)u_n^2dx= \int_{\Omega_r}V(x)u_{r,s}^2dx+o(1), 
				\qquad	B(u_n)=B(u_{r,s})+o(1),
				\end{equation}
				then one can yield 
				$u_n\to u_{r,s}$ strongly in $H^1_0(\Omega_r)$ and  $E_{r,s}(u_{r,s})=m_{r,s}(a)$.
				\vskip 0,1in
				Now we show $u_{r,s}\ge 0$.
				The strategy is inspired by \cite{BCJS}. For fixed $a\in(0,a_0)$, let $m_{r,s}:=m_{r,s}(a)$. It is easy to see the function $s\mapsto m_{r,s}$ is nonincreasing so that one can define the derivative $m'_{
					r,s}$ for almost every $s\in[\frac{1}{2},1]$. Set $S^{\circ}:=\{s\in [\frac{1}{2},1]:m'_{r,s}~\text{exists}\},$  then $\Big|[\frac{1}{2},1]\setminus S^{\circ}\Big|=0$. For fixed $s\in [\frac{1}{2},1]$, we can choose a monotone increasing sequence $\{s_n\}\subset [\frac{1}{2},1]$ with $s_n\rightarrow s$ in $\R$.  Then adapting the similar argument in proving Theorem 1.10 in \cite{BCJS}, there exist $\{\gamma_n\}\subset\Gamma_{r,a}$ and $K:= K(m'_{r,s})$
				such that for any $t\in[0,1]$:
				\begin{description}
					\item[(1)] $\int_{\Omega_r}|\nabla \gamma_{n}(t)|^2dx\leq K$ whenever $E_{r,s}(\gamma_{n}(t))\geq m_{r,s}-(2-m'_{r,s})(s-s_n)$.
					\item[(2)] $\max\limits_{t\in[0,1]}E_{r,s}(\gamma_{n}(t))\leq m_{r,s}-(2-m'_{r,s})(s-s_n)$.
				\end{description}
				For any $t\in[0,1]$, let $\tilde{\gamma}_{n}(t):=|\gamma_{n}(t)|$. Then $\{\tilde{\gamma}_{n}\}\subset \Gamma_{r,a} $. Moreover, by \cite{LL2001}, one has $\int_{\Omega_r}|\nabla \tilde{\gamma}_{n}(t)|^2dx\leq \int_{\Omega_r}|\nabla \gamma_{n}(t)|^2dx$. Thus we have
				\begin{description}
					\item[(a)] if $E_{r,s}(\tilde{\gamma}_{n}(t))\geq m_{r,s}-(2-m'_{r,s})(s-s_n)$, then $E_{r,s}(\gamma_{n}(t))\geq m_{r,s}-(2-m'_{r,s})(s-s_n)$. By \textbf{(1)}, we have $\int_{\Omega_r}|\nabla \gamma_{n}(t)|^2dx\leq K$ which yields $\int_{\Omega_r}|\nabla \tilde{\gamma}_{n}(t)|^2dx\leq K$. 
					\item[(b)] $\max\limits_{t\in[0,1]}E_{r,s}(\tilde{\gamma}_{n}(t))\leq \max\limits_{t\in[0,1]}E_{r,s}(\gamma_{n}(t))\leq m_{r,s}-(2-m'_{r,s})(s-s_n)$.
				\end{description} 
				\vskip  0.15in
				\textbf{(a)}-\textbf{(b)} indicate that 	\textbf{(1)}-	\textbf{(2)}   hold for $\tilde{\gamma}_{n}$. Now we can replace $\gamma_{n}$ with 	 $\tilde{\gamma}_{n}$ in the proof of \cite[Theorem 1.5]{BCJS}, thus we obtain
				a nonnegative bounded $(PS)$ sequence $\{u_n\}$, as a consequence $u_{r}\geq0$.
				\color{black}   
			\end{proof}
			In the next lemma, we will prove the Pohozaev identity, which is satisfied by the weak solutions to Problem (\ref{eq-Omegar-s}) for any $s\in[\frac{1}{2},1]$. This identity will be useful to pass to the limit as $s\to 1$ and obtain a solution to Problem (\ref{eq-Omega_r}).
			\begin{lemma}[Pohozaev identity]
			\label{lemma-pohozaev}
            Assume that either $(V_1)$ holds or $V\equiv 0$. Let $a\in(0,a_0)$ and $r_a>0$ be given in Lemma \ref{lemma-MP-geom}. For any $s\in[\frac{1}{2},1]$ and $r\in [r_a,\infty]$, let $u\in \mathcal{S}_{r,a}$ be a weak solution to (\ref{eq-Omegar-s}), where $\Omega_\infty:=\R^3$ and $\mathcal{S}_{\infty,a}:=\mathcal{S}_a$. Then it satisfies the Pohozaev identity
				\begin{equation}\label{Pohozaev}
					P_{r,s}(u):=\int_{\Omega_r}|\nabla u|^2dx+\frac{1}{4}B(u)-\frac{3(p-2)s}{2p}\int_{\Omega_r}|u|^pdx-\int_{\Omega_r}\nabla V(x)\cdotp x u^2dx=0.
				\end{equation}
			\end{lemma}
            We note that, if $V=0$, then $P_{\infty,1}=P$, where $P$ is defined in the Introduction.
        \color{black}\begin{proof}
				Given a weak solution $u:=u_{r,s}\in \mathcal{S}_{r,a}$ to (\ref{eq-Omegar-s}) and a real number $t\in(1/2,3/2)$. On the one hand, we observe that $u^t(x)=t^{3/2}u(tx)\in \mathcal{S}_{\frac{r}{t},a}\subset \mathcal{S}_{2r,a}$ and
				\begin{equation}\notag
					\begin{aligned}
						\frac{d}{dt}\bigg|_{t=1}E_{\frac{r}{t},s}(u_{t})&=\frac{d}{dt}\bigg|_{t=1}E_{2r,s}(u^{t})=\nabla E_{2r,s}(u)[(\partial_t u^t)|_{t=1}]
						=\nabla E_{r,s}(u)[\partial_t u^t|_{t=1}]\\
						&=\lambda\int_{\Omega_r}u(\partial_t u^t|_{t=1})
						=\frac{\lambda}{2}\frac{d}{dt}\bigg|_{t=1}\left(\int_{\Omega_{\frac{r}{t}}} (u^t)^2 dx\right)=0.
					\end{aligned}
				\end{equation}
				On the other hand, a direct computation based on a change of variables shows that
				$$\frac{d}{dt}\bigg|_{t=1}E_{\frac{r}{t},s}(u^{t})=\int_{\Omega_r}|\nabla u|^2dx+\frac{1}{4}B(u)-\frac{3(p-2)s}{2p}\int_{\Omega_r}|u|^pdx-\int_{\Omega_r}\nabla V(x)\cdotp x u^2dx,$$
				which concludes the proof.
			\end{proof}
			\color{black}
            
            In Proposition \ref{prop-sol-Omegar-s}, we obtain a normalized  solution $(\lambda_{r,s},u_{r,s})$ to Problem (\ref{eq-Omegar-s}) in $\Omega_r$. The Pohozaev identity enables us to prove a bound for the $L^2(\Omega_r)$-norm of the gradient of $u_{r,s}$ which is uniform in $s$, so that we will be able to take the limit as $s\to 1$.
			\begin{lemma}\label{lemma-unif-bound-s}
				Assume that $(V_1)$ holds. Let $s\in[\frac{1}{2},1]$ be such that the solution  $(\lambda_{r,s},u_{r,s})$ to Problem (\ref{eq-Omegar-s}) constructed in Proposition \ref{prop-sol-Omegar-s} exists. Then we have
				\begin{equation}
					\label{bound-u_rs}
					\int_{\Omega_r}|\nabla u_{r,s}|^2dx\le \frac{6(p-2)}{3p-10}m_{r,s}(a)+\frac{4a^2}{3p-10}\|W\|_\infty,
				\end{equation}
				for any $r\geq r_a$, $s\in[\frac{1}{2},1]$.
			\end{lemma}
			\color{black}\begin{proof}
				In the proof we will set $u:=u_{r,s}$. By the Pohozaev identity (\ref{Pohozaev}) and the fact that $m_{r,s}(a)=E_{r,s}(u_{r,s})$ we have
				\begin{equation}\notag
					\begin{aligned}
						\int_{\Omega_r}|\nabla u|^2dx&=\int_{\Omega_r}\nabla V(x)\cdotp x u^2 dx+\frac{3(p-2)s}{2p}\int_{\Omega_r}|u|^pdx-\frac{1}{4}B(u)\\
						&=\int_{\Omega_r}\nabla V(x)\cdotp x u^2 dx-\frac{1}{4}B(u)\\
						&+\frac{3(p-2)}{2}\left(\frac{1}{2}\int_{\Omega_r}|\nabla u|^2dx+\frac{1}{4}B(u)-m_{r,s}(a)+\frac{1}{2}\int_{\Omega_r}V(x)u^2 dx\right)\\
						&\ge \frac{3(p-2)}{4}\int_{\Omega_r}|\nabla u|^2dx+\int_{\Omega_r}\nabla V(x)\cdotp x u^2 dx+\frac{3(p-2)}{4}\int_{\Omega_r}V(x)u^2 dx-\frac{3(p-2)}{2}m_{r,s}(a).
					\end{aligned}
				\end{equation}
				Using that $V(x)\ge 0$ we have the required estimate.
			\end{proof}
			\begin{remark}\label{rem-unif-est-u-rs}
				Due the upper bound for the mountain-pass level $m_{r,s}(a)$ given in Lemma \ref{lemma-MP-geom}, we have the uniform estimate
				\begin{equation}
					\label{bound-u_rs-unif-s}
					\int_{\Omega_r}|\nabla u_{r,s}|^2dx\le \frac{6(p-2)}{3p-10}\tilde{c}_a+\frac{4a^2}{3p-10}\|W\|_\infty,
				\end{equation}
				for any $r\geq r_a$, $s\in[\frac{1}{2},1]$, where $\tilde{c}_a$ is defined in Lemma \ref{lemma-MP-geom} independent of $s$ and $r$.
			\end{remark}
			\begin{proposition}\label{prop-existence-Omega-r}
				Assume that $(V_1)$ holds. Then for any $a\in(0,a_0)$ and $r\ge r_a$ there exists a solution $(\lambda_r,u_r)\in\R\times H^1_0(\Omega_r)$ to Problem (\ref{eq-Omega_r}) such that $E_{r,1}(u_r)=m_{r,1}(a)$, $u_r\ge 0$ fulfilling
				\begin{equation}
					\label{bound-u_r1}
					\int_{\Omega_r}|\nabla u_{r}|^2dx\le \frac{6(p-2)}{3p-10}(1+\theta)c_a+\frac{4a^2}{3p-10}\|W\|_\infty,\qquad\forall r\ge \tilde{r}_a
				\end{equation}
				for some $\tilde{r}_a>r_a$ ($r_a$ is given in Lemma \ref{lemma-MP-geom}).
			\end{proposition}
			\begin{proof}
				By Proposition \ref{prop-sol-Omegar-s}, for almost every $s\in[\frac{1}{2},1]$, there exists a solution $(\lambda_{r,s},u_{r,s})$ to Problem (\ref{eq-Omegar-s}) with $u_{r,s}\ge 0$ fulfilling (\ref{bound-u_rs}). By Remark \ref{rem-unif-est-u-rs}, the $H^1(\Omega_r)$-norm of $u_{r,s}$ is bounded uniformly in $s$ and the same is true for the Lagrange multipliers $\lambda_{r,s}$, so that there exist a sequence $s_n\to 1$ and a couple $(\lambda_r,u_r)\in \R\times H^1_0(\Omega_r)$ such that as $n\to\infty$,
				$$\lambda_{r,s_n}\to\lambda_r ~\text{in}~ \R,\quad u_{r,s_n}\rightharpoonup u_r~\text{weakly\,in}\,H^1(\Omega_r).$$
				 In particular the equation
				$$-\Delta u_r+\lambda_r u_r+V(x)u_r+\phi_{u_r} u_r=|u_r|^{p-2}u_r\quad\text{in}\quad\Omega_r$$
				is satisfied and, using the weak lower semicontinuity of the norm and taking the limit as $s\to 1$ in (\ref{bound-u_rs-unif-s}), we can see that
				\begin{equation}\notag
					\begin{aligned}
						\|\nabla u_r\|^2_2\le\liminf_{s\to 1}\|\nabla u_{r,s}\|^2_2&\le\frac{6(p-2)}{3p-10}\lim_{s\to 1}m_{r,1}(a)+\frac{4a^2}{3p-10}\|W\|_\infty\\
                        &\le\frac{6(p-2)}{3p-10}(1+\theta)c_a+\frac{4a^2}{3p-10}\|W\|_\infty,
					\end{aligned}
				\end{equation}
				that is $u_r$ fulfills (\ref{bound-u_r1}).
				By the compactness of the embeddings $H^1_0(\Omega_r)\subset L^q(\Omega_r)$ for $2\le q<6$, we can see that $u_r\in \mathcal{S}_{r,a}$, $u_r\ge 0$ and $u_{r,s_n}\to u_r$ strongly in $H^1(\Omega_r)$, which yields that $$m_{r,s}(a)=E_{r,s}(u_{r,s})\to E_{r,1}(u_r)=m_{r,1}(a).$$
			\end{proof}
			\begin{remark}
				Proposition \ref{prop-existence-Omega-r} concludes the proof of Theorem \ref{th-bd-V_1}. Moreover, (\ref{bound-u_r1}) shows that $\{u_r\}_{r\geq r_a}$ is bounded in $H^1(\R^3)$ uniformly in $r$.
			\end{remark}

			\subsection{Passing to the limit as $r\to\infty$}\label{sec-lim-V_1}
			We note that the upper bound (\ref{bound-u_r1}) is uniform in $r$. As a consequence $u_r$ is bounded in $H^1(\R^3)$. Therefore, passing to the limit as $r\to\infty$, we can prove the following existence result.
			\begin{proposition}\label{prop-existence-sol}
				Assume that $(V_1)$ holds. For any $a\in(0,a_0)$ and $r\geq r_a$, let $(\lambda_r, u_r)\in\R\times H^1_0(\Omega_r)$ be the solution constructed in Theorem \ref{th-bd-V_1} with $\Omega:=B_1(0)$. Then
				there exist $a^*\in(0,a_0)$ and an increasing sequence $r_n\to\infty$ such that for any $a\in(0,a^*)$,
				\begin{center}
					$u_{r_n}\rightharpoonup u$ weakly in $H^1(\R^3)$ and $\lambda_{r_n}\to\lambda$ in $\R$,
				\end{center}  where $(\lambda,u)\in\R\times H^1(\R^3)$ is the solution to the equation
				\begin{equation}\label{eq-R3}
					-\Delta u+\lambda u+V(x)u+\phi_u u=|u|^{p-2}u\quad\text{in}\quad\R^3
				\end{equation}
				with $\lambda>0$ and $u\ge 0$. 

			\end{proposition}
			\color{black}
			\begin{proof}
				For $r\geq r_a$, we set 
				\begin{equation}\label{ABCDEr}
					\begin{aligned}
						&A_r:=\int_{\Omega_r}|\nabla u_r|^2dx, ~B_r:=B(u_r),~C_r:=\int_{\Omega_r}V(x)u_r^2dx,\\
						&D_r:=\int_{\Omega_r}\nabla V(x)\cdot x u^2_r dx,~E_r:=\int_{\Omega_r}|u_r|^pdx.
					\end{aligned}
				\end{equation}
				Using that $E_{r,1}(u_r)=m_{r,1}(a)$, the Pohozaev identity (\ref{Pohozaev}) and testing equation (\ref{eq-Omega_r}) with $u_r$, we have
				\begin{equation}\label{syst-ABCDEr}
					\begin{aligned}
						&A_r+C_r+\frac{1}{2}B_r-\frac{2}{p}E_r=2m_{r,1}(a),\\
						&A_r+\frac{1}{4}B_r-\frac{3(p-2)}{2p}E_r-D_r=0,\\
						&A_r+C_r+\lambda_r a^2+B_r=E_r.
					\end{aligned}
				\end{equation}
				Due to the upper estimate (\ref{bound-u_r1}), $A_r$ is bounded uniformly in $r$, 
			which yields $u_r$ is bounded uniformly in $r$ in $H^1(\R^3)$. By the Sobolev embeddings, \eqref{local-term} and $(V_1)$, one has
				\begin{equation}\label{est-Br-black}
					B_r\le \bar{C}A_r^{1/2}\|u_r\|^3_2=\bar{C}A_r^{1/2}a^3\le \tilde{C} a^3 c_a^{1/2},
				\end{equation}
				for some constant $\tilde{C}>0$, hence in particular we can see that $B_r$ is bounded uniformly in $r$. By the Sobolev embeddings and system (\ref{syst-ABCDEr}), it is possible to see that $C_r,\,D_r,\,E_r$ and 
				$\lambda_r$ are also bounded in $\R$, which yields the existence of an increasing sequence $r_n\to\infty$ such that
				$$A_{r_n}\to A,~B_{r_n}\to B,~C_{r_n}\to C,~D_{r_n}\to D,~E_{r_n}\to E,~\lambda_{r_n}\to\lambda\quad\text{in}\quad\R
				$$
				and
				\begin{equation*}
					u_{r_n}\rightharpoonup u\quad\text{in}\quad H^1(\R^3)
				\end{equation*}
				for some $A,\,B,\,C,\,D,\,E,\,\lambda\in\R$ and $u\in H^1(\R^3)$.  Moreover, one can see that $u\geq0$ is a weak solution to (\ref{eq-R3}) (here we use the fact $u_r\geq0$).\\
				
				It remains to prove that $\lambda>0$. Subtracting the second relation in (\ref{syst-ABCDEr}) to the first one, we have
				\begin{equation}\notag
					\begin{aligned}
						(3p-10)a^2\lambda_r&=(3p-10)\left(\frac{p-2}{p}E_r-\frac{1}{2}B_r-2m_{r,1}(a)\right)\\
						&=2(p-2)\big(2m_{r,1}(a)-C_r-D_r-\frac{1}{4}B_r\big)-(3p-10)\left(\frac{1}{2}B_r+2m_{r,1}(a)\right)\\
						&=2(6-p)m_{r,1}(a)-2(p-2)C_r-2(p-2)D_r-4(p-3)B_r\\
						&\ge 2(6-p)c_a-2(p-2)C_r-2(p-2)D_r-4(p-3)B_r.
					\end{aligned}
				\end{equation}

		We claim: there exists $a^*\in(0,a_0)$ such that  for any $a\in(0,a^*)$, there exist  $\delta>0$ and $\bar{r}_a>r_a$ such that
				\begin{equation}\label{lower-est-lambda-r-black}
					(3p-10)a^2\lambda_r>\delta c_a,\quad\forall ~r>\bar{r}_a.
				\end{equation}
		Due to (\ref{est-V-infty})  and (\ref{est-Br-black}), there is $\bar{r}_a>r_a$ large enough such that
				\begin{equation}\notag
					\begin{aligned}
						2(p-2)(C_r+|D_r|)+4(p-3)B_r&\le 2(p-2)(\|V\|_\infty a^2+\|W\|_\infty a^2)+\tilde{C}a^3 c_a^{1/2}\\
						&< 2(p-2)(2\theta+\eta+\tilde{C}a^3 c_a^{-1/2})c_a\\
						&<2(6-p-\delta)c_a
					\end{aligned}
				\end{equation}
				if $\delta>0$ is small enough and $a\in(0,a^*)$ with 
				$a^*>0$ small enough (here we make use of the fact $a\mapsto c_a$ is nonincreasing in a right neighbourhood of $0$, see \cite[Theorem 1.2]{B2013}). This concludes the proof of the claim.\\
				
			Take the limit as $r\to\infty$ in (\ref{lower-est-lambda-r-black}), which shows that $$(3p-10)a^2\lambda\ge \delta c_a>0,\quad\forall\,0<a<a^*.$$
			This completes the proof.
			\end{proof}
			\begin{remark}
				\label{rem-lambda-r>0}
			From	Proposition \ref{prop-existence-sol}, we have the following remarks:
				\begin{itemize}
					\item The proof of Proposition \ref{prop-existence-sol} shows that $\lambda_r$ is positive and bounded uniformly in $r$ from above and from below. More precisely $$0<\liminf_{r\to\infty} \lambda_r\le\limsup_{r\to\infty} \lambda_r<\infty\qquad\forall\,0<a<a^*.$$
					This proves (\ref{bound-lambda}) in Theorem \ref{th-unif-bound}.
					\item In Proposition \ref{prop-existence-sol} we take $\Omega:=B_1(0)$ because we want the property $\cup_{r>r_\rho}\Omega_r=\R^3$ to be satisfied and the boundary to be Lipschitz, in order to have $\nabla u_r\cdotp y\in H^1(\Omega_r)$ for any $y\in\R^3$, as we will explain below. Moreover, this choice simplifies the forthcoming computations in the proof of Theorem \ref{theorem1.2}. In any case, Proposition \ref{prop-existence-sol} is true in any bounded convex Lipschitz domain $\Omega$ containing $0$.
					\item We note that at this level we have not proved that $u\in \mathcal{S}_a$ yet.
				\end{itemize}
			\end{remark}
			In the sequel we will be interested in considering the limit of our solutions $u_r$ as $r\to\infty$, at least when $\Omega$ is a ball. In order to do so, we will need to test equation (\ref{eq-Omega_r}) with $\nabla u_r\cdotp y$ whenever $y\in\Omega_r$ is fixed. For this reason we need to prove that the solutions constructed in Theorem \ref{th-bd-V_1} fulfill $\nabla u_r\cdotp y\in H^1(\Omega_r)$ for any $y\in\Omega_r$. This follows from the two forthcoming results.
			\begin{proposition}\label{prop-L^q}
				Let $\Omega\subset\R^3$ be a bounded convex domain. 
                Let $u\ge 0$ be a subsolution to the problem
				\begin{equation}\label{eq-bound-L^p}
					\begin{cases}
						-\Delta u
                        \le u^{p-1}\quad \text{in }\quad \Omega,\\
						u\in H^1_0(\Omega).
					\end{cases}
				\end{equation}
				with $p\in(2,6)$. Then $u\in L^{\bar{q}}(\Omega)$ for any $\bar{q}\in(6,\infty]$ and $$\|u\|_{L^{\bar{q}}(\Omega)}\le C\max\{\|u\|_6,\|u\|_6^{1+\frac{p-2}{6-p}}\},$$ 
				for some $C=C(p)>0$ depending on $p$ but not on $\Omega$ and $\bar{q}$.
			\end{proposition}
			\begin{remark}
                The proof of        Proposition \ref{prop-L^q} is very similar to the one of \cite[Theorem 1.3]{PWZ}. In any case, here we stress that such a technique provides a bound of the $L^{\bar{q}}(\Omega)$ norm which is uniform in $\Omega$ and $\bar{q}$. Such uniformity 
				will play a crucial role in the proof of Theorem \ref{theorem1.2} when considering the limit as $r\to\infty$.
			\end{remark}
			\begin{proof}[Proof of Proposition \ref{prop-L^q}]
				First we set, for $M>0$ and $x\in\Omega$, $v_M(x):=\min\{u(x),M\}$ and we test equation (\ref{eq-bound-L^p}) with $v:=v_M^{2\chi+1}$, where $\chi>0$ will be fixed below. A direct computation shows that
				$$\int_\Omega (2\chi+1)v_M^{2\chi}|\nabla v_M|^2 dx
                \le\int_\Omega u^{p-1}v dx.$$ 
                Using the Sobolev embedding $H^1_0(\Omega)\subset L^6(\Omega)$, we can see that
				\begin{equation}\notag
					\begin{aligned}
						&\int_\Omega u^{p-1}v dx\ge \int_\Omega (2\chi+1)v_M^{2\chi}|\nabla v_M|^2 dx\\
						&=\frac{2\chi+1}{(\chi+1)^2}\int_\Omega |\nabla (v_M^{\chi+1})|^2 dx\ge S\frac{2\chi+1}{(\chi+1)^2}\|v_M\|^{2(\chi+1)}_{6(\chi+1)}.
					\end{aligned}
				\end{equation}
				On the other hand, by the definition of $v_M$ and the Holder inequality, one have
				\begin{equation}\notag
					\int_\Omega u^{p-1}v dx\le\int_\Omega u^{p-2}u^{2(\chi+1)} dx\le \|u\|_6^{p-2}\|u\|^{2(\chi+1)}_{\gamma_0(\chi+1)},
				\end{equation}
				where we have set $\gamma_0:=\frac{12}{8-p}\in(2,6)$ since $p\in(2,6)$. As a consequence, we have the estimate
				$$\|v_M\|_{6(\chi+1)}\le \left(\frac{\chi+1}{(S(2\chi+1))^{1/2}}\right)^{\frac{1}{\chi+1}}\|u\|_6^{\frac{p-2}{2(\chi+1)}}\|u\|_{\gamma_0(\chi+1)}\qquad\forall\,M>0,\,\chi>0.$$
				Using the fact that $u\ge 0$, $v_M\to u$ as $M\to\infty$ point-wise in $\Omega$ and the Fatou lemma,  we have as $M\to\infty$,
				\begin{equation}\label{est-u-chi}
					\|u\|_{6(\chi+1)}\le \left(\frac{\chi+1}{(S(2\chi+1))^{1/2}}\right)^{\frac{1}{\chi+1}}\|u\|_6^{\frac{p-2}{2(\chi+1)}}\|u\|_{\gamma_0(\chi+1)}\qquad\forall\,\chi>0. 
				\end{equation}
				Applying (\ref{est-u-chi}) with $\chi=\chi_1>0$ such that $1+\chi_1=\gamma_0/6$ we have
				$$\|u\|_{6(\chi_1+1)}\le \left(\frac{\chi_1+1}{(S(2\chi_1+1))^{1/2}}\right)^{\frac{1}{\chi_1+1}}\|u\|_6^{1+\frac{p-2}{2(\chi_1+1)}}.$$
				By induction, for any $n>0$ we can choose $\chi_n>0$ such that $1+\chi_n=(6/\gamma_0)^n$, so that, by \eqref{est-u-chi}
				$$\|u\|_{6(\chi_n+1)}\le S^{-\frac{1}{2}\sum_{k=1}^n\frac{1}{\chi_k+1}}\prod_{k=1}^n\left(\frac{\chi_k+1}{(2\chi_k+1)^{1/2}}\right)^{\frac{1}{\chi_k+1}}\|u\|_6^{1+\frac{p-2}{2}\left(\sum_{k=1}^n\frac{1}{\chi_k+1}\right)},$$
				for any $n\ge 1$. We note that $\chi_n\to\infty$. Using the properties of the function $\varphi(y):=\left(\frac{y+1}{(2y+1)^{1/2}}\right)^{\frac{1}{\sqrt{y+1}}}$ and the properties of the geometric series we conclude that
				$$\|u\|_{6(\chi_n+1)}\le c\|u\|_6^{1+\frac{p-2}{6-p}(1-(\gamma_0/6)^n)}\qquad\forall\, n\ge 1,$$
				for some constant $c=c(p)>0$ depending on $p$ but not on $\Omega$. We refer to the proof of \cite[Theorem 1.3]{PWZ} for the details.\\
				
				Taking $\bar{q}>6$, $n\ge 1$ such that $6(\chi_n+1)>\bar{q}$ and $\alpha\in(0,1)$ defined by $$\frac{1}{\bar{q}}=\frac{\alpha}{6}+\frac{1-\alpha}{6(\chi_n+1)},$$  
				an interpolation inequality gives 
				$$\|u\|_{\bar{q}}\le \|u\|_6^\alpha\|u\|^{1-\alpha}_{6(\chi_n+1)} \le c\|u\|_6^{1+\frac{p-2}{6-p}(1-(\gamma_0/6)^n)(1-\alpha)}\le c\max\{\|u\|_6,\|u\|_6^{1+\frac{p-2}{6-p}}\}=:g(\|u\|_6),$$
				where $c>0$ is independent of $q$ and $\Omega$. This concludes the proof for $\bar{q}\in[6,\infty)$.\\
				
				In orded to treat the case $q=\infty$, we note that the function $v(x):=\min\{u(x),g(\|u\|_6)+1\}$ is bounded in $\Omega$ and satisfies $\|v\|_{\bar{q}}\le\|u\|_{\bar{q}}\le g(\|u\|_6)$ for any $\bar{q}\ge 6$, hence
				$$\|v\|_\infty=\lim_{\bar{q}\to\infty}\|v\|_{\bar{q}}\le g(\|u\|_6).$$
				This yields that $v(x)=\min\{u(x),g(\|u\|_6)+1\}\le g(\|u\|_6)$, hence $v=u$, so in particular $u\in L^\infty(\Omega)$ and $\|u\|_\infty\le g(\|u\|_6)$.
			\end{proof}
			\begin{remark}\label{rem-bound-u-infty}
				In particular, since $\liminf_{r\to\infty}\lambda_r>0$ (see Remark \ref{rem-lambda-r>0}), Proposition \ref{prop-L^q} shows that the solutions $(\lambda_r,u_r)$ constructed in Theorem \ref{th-bd-V_1} are bounded for $r$ large enough and fulfill $$\sup_{\bar{q}\in[6,\infty]}\sup_{r\geq r_a}\|u_r\|_{\bar{q}}<\infty,$$
				which proves (\ref{bound-u-r-infty}) in Theorem \ref{th-unif-bound}. In fact, due to the uniform bound in $H^1(\Omega_r)$ given by (\ref{bound-u_r1}) and the Sobolev embedding $H^1(\R^3)\subset L^6(\R^3)$, we have
				$$\|u_r\|_{\bar{q}}\le C(p)\max\{\|u_r\|_6,\|u_r\|_6^{1+\frac{p-2}{6-p}}\}\le C(p)\max\{\|u_r\|_{H^1(\Omega_r)},\|u_r\|_{H^1(\Omega_r)}^{1+\frac{p-2}{6-p}}\}\le M_{a,p},\qquad\forall\,\bar{q}\in[6,\infty],$$
				for some constant $M_{a,p}>0$ depending on $a$ and $p$ only.
			\end{remark}
            In particular, the uniform bound of $u_r$ given by Remark \ref{rem-bound-u-infty} gives a uniform bound for $\phi_{u_r}$ and $\nabla\phi_{u_r}$ in $L^\infty(\Omega_r)$.
			\begin{lemma}\label{lemma-est-nabla-phi}
				Assume that $\partial\Omega$ is Lipschitz. Then for any $a\in(0,a^*)$, where $a^*$ is given in Proposition \ref{prop-existence-sol}, there exists a constant $M_{a}>0$ such that for $r\geq r_a$, we have $\phi_{u_r}\in W^{1,\infty}(\R^3)\cap H^2_{loc}(\R^3)$ and
				\begin{equation}
					\label{est-phi-infty}
					\|\nabla\phi_{u_r}\|_\infty+\|\phi_{u_r}\|_\infty\le M_a
				\end{equation}
			\end{lemma}
			\begin{proof}
				The proof of (\ref{est-phi-infty}) relies on the elliptic estimates. In fact, for any $x\in\R^3$ we have $\phi_{u_r}\in W^{2,6}(B_1(x))$ (see Theorem $9.11$ of \cite{Tru}) and
				\begin{equation}\notag
					\begin{aligned} 
						\|\phi_{u_r}\|_{W^{2,6}(B_1(x))}&\le c(\|\phi_{u_r}\|_{L^6(B_2(x))}+\|\Delta\phi_{u_r}\|_{L^6(B_2(x))})=c(\|\nabla\phi_{u_r}\|_{L^2(\R^3)}+\|u_r\|^2_{L^{12}(B_2(x))})\\
						&\le c(\|u_r\|^2_{12/5}+\|u_r\|^2_{12})\le  c(\|\nabla u_r\|^2_{2}+\max\{\|u_r\|_6,\|u_r\|_6^{1+\frac{p-2}{6-p}}\})\le M_a
					\end{aligned}
				\end{equation}
				thanks to Propositions \ref{prop-L^q}, the uniform bound given by (\ref{bound-u_r1}) and the fact that
                $$\|\nabla\phi_{u_r}\|^2_2\le\|\phi_{u_r}\|_6\|u_r\|^2_{12/5}\le S^{-1/2}\|\nabla\phi_{u_r}\|_2\|u_r\|^2_{12/5}.$$
                The result follows from the Sobolev embedding $W^{2,6}(B_1(x))\subset C^{1,\frac{1}{2}}(B_1(x))$ and the fact that $W^{2,6}_{loc}(\R^3)\subset H^2_{loc}(\R^3)$.\\
				
			\end{proof}
        We note that the estimates provided by Lemma \ref{lemma-est-nabla-phi} are uniform in $r$, since the bound of the $L^q(\Omega_r)$-norm of $u_r$ is uniform in $r$. This will be crucial in the proof of Theorem \ref{theorem1.2}.\\
            
	Now we apply 
        Proposition \ref{prop-L^q} to prove the regularity of $u_r$ constructed in Theorem \ref{th-bd-V_1}  if the domain $\Omega$ is Lipschitz. 
			\begin{proposition}\label{prop-L-infty}
				Assume that $\Omega$ is a convex bounded Lipschitz domain and the hypothesis of Theorem \ref{th-bd-V_1} are satisfied. Then the solutions $u_r$ constructed in Theorem \ref{th-bd-V_1}  are in $C^{1}(\bar{\Omega}_r)\cap W^{2,2}(\Omega_r)$.
			\end{proposition}
			\begin{proof}
				Let $\bar{p}\in(3,\infty)$. Due to Proposition \ref{prop-L^q} and the fact that $V\in L^\infty(\R^3)$ and $\phi_{u_r}\in L^\infty(\R^3)$, we can see that $\Delta u_r\in L^{\bar{p}}(\Omega_r)$. Therefore, since $\Omega$ Lipschitz and $u_r\in H^1_0(\Omega_r)$, the elliptic estimates give that $u_r\in W^{2,\bar{p}}(\Omega_r)$. Thus using the Sobolev embedding  $W^{2,\bar{p}}(\Omega_r)\subset C^{1,1-\frac{3}{\bar{p}}}(\Omega_r)$ and the fact that $W^{2,\bar{p}}(\Omega_r)\subset W^{2,2}(\Omega_r)$, we conclude the proof. 
			\end{proof}
			\begin{remark}\label{rem-test-function}
				Proposition \ref{prop-L-infty} shows that, if $\Omega$ is Lipschitz, then the solution constructed in Theorem \ref{th-bd-V_1} fulfills $\nabla u_r\cdotp y\in H^1(\Omega_r)$ for any $y\in\Omega_r$, hence it can be used as a test function (see below). 
			\end{remark}

			\begin{lemma}\label{lemma-PS}
Let $\Omega:=B_1(0)$, $a\in(0,a^*)$ and the sequence $\{u_{r_n}\}$ be constructed in Proposition \ref{prop-existence-sol}. Then $\{u_{r_n}\}$ is a bounded $(PS)$ sequence of $E_{\infty,1}$ constraint to $\mathcal{S}_a$.
			\end{lemma}
			\begin{proof}
				Let $(\lambda_{r,s},u_{r,s})$ be the solution of Problem (\ref{eq-Omegar-s}) obtained in Proposition \ref{prop-existence-Omega-r}. Due to Lemma \ref{lemma-MP-geom}, one has as $r\to\infty$,
                $$E_{r,1}(u_r)= E_{\infty,1}(u_r)+o_r(1)=m_{r,1}(a)+o_r(1)\in[c_a+1,\tilde{c}_a+1].$$
		We claim $\nabla_{\mathcal{S}_a} E_{\infty,1}(u_r)=o_r(1)$ as $r\to\infty$ in $H^{-1}(\R^3)$. In fact, taking a compactly supported test function $\varphi\in C^\infty_c(\R^3)$, we have $$\nabla _{\mathcal{S}_a}E_{\infty,1}(u_r)[\varphi]=\nabla _{\mathcal{S}_a}E_{r,1}(u_r)[\varphi]=0$$
				if $r$ is so large that ${\rm supp}(\varphi)\subset B_r(0)$, since $u_r$ is a solution to equation (\ref{eq-Omega_r}) in $B_r(0)$. If $v\in H^1(\R^3)$, then for any $\epsilon>0$ there exists $\varphi\in C^\infty_c(\R^3)$ such that $\|v-\varphi\|_{H^1(\R^3)}<\epsilon$. Then, taking $r>0$ so large that ${\rm supp}(\varphi)\subset B_r(0)$, we have
				$$|\nabla _{\mathcal{S}_a}E_{\infty,1}(u_r)[v]|=|\nabla _{\mathcal{S}_a}E_{\infty,1}(u_r)[v-\varphi]|\le \tilde{C}_a\epsilon,$$
				for some $\tilde{C}_a>0$, here we use the fact that $\nabla_{\mathcal{S}_a} E_{\infty,1}(u_r)$ is bounded in $H^{-1}(\R^3)$ uniformly in $r$ (indeed, for $r$ large enough, it holds
				\begin{equation}\notag
					\begin{aligned}
						|\nabla_{\mathcal{S}_a} E_{\infty,1}(u_r)[v]|&=\bigg|\int_{B_r(0)}\nabla u_r\cdotp\nabla v dx+\int_{B_r(0)}(V(x)+\lambda_r)u_r v dx\\
						&+\int_{B_r(0)}\phi_{u_r}u_r  v dx-\int_{B_r(0)}u_r^{p-1}v dx\bigg|\\
						&\le\|\nabla u_r\|_2\|\nabla v\|_2+\left(\|V\|_\infty+2\limsup_{r\to\infty}\lambda_r\right)a\|v\|_2\\
						&+c\|u_r\|^3_{H^1(\R^3)}\|v\|_{H^1(\R^3)}+\|u_r\|_6^{p-1}\|v\|_2
					\end{aligned}
				\end{equation}
				and $(\lambda_r,u_r)$ is bounded in $\R\times H^1(\R^3)$ uniformly in $r$ (see Proposition \ref{prop-existence-sol})).
			\end{proof}
			Now we state the following lemma, which will be used to analyze the behavior of the $(PS)$ sequences as $n\to\infty$. We recall that, in our notations, we have 
			$$J_{V,\lambda}(u)=E_{\infty,1}(u)+\frac{\lambda}{2}\int_{\R^3}u^2 dx,\quad I_{\lambda}(u)=\bar{E}_{\infty,1}(u)+\frac{\lambda}{2}\int_{\R^3}u^2 dx\quad\forall\quad u\in H^1(\R^3).$$
			\begin{lemma}[Splitting Lemma]\label{lem2.2}
				Let $V\in L^{\bar{q}}(\R^3)$, for some $\bar{q}\in\left[\frac{3}{2},\infty\right]$. Assume furthermore that $V^-:=-\min\{V,0\}\in L^{\frac{3}{2}}(\R^3)$ with $\|V^-\|_{\frac{3}{2}}<S$ and $\lim_{|x|\to\infty}V(x)=0$. Let $\{u_n\}\subset H^1(\mathbb{R}^3)$ be a bounded $(PS)$ sequence for $J_{V,\lambda}$ such that $u_n\rightharpoonup u$ in $H^1(\mathbb{R}^3)$. Then there exists an integer $k\geq0$, $k$ non-trivial solutions $w^1,\cdots,w^k\in H^1(\mathbb{R}^3)$ to the limit equation
				\begin{equation}\label{lim-eq}
					-\Delta u+\lambda u+(|x|^{-1}\ast |u|^2) u=|u|^{p-2}u\quad\text{in}\quad\R^3
				\end{equation}
				and $k$ sequences $\{y_n^j\}_n\subset \mathbb{R}^3,1\leq j\leq k$ such that $|y_n^j|\rightarrow\infty$ as $n\rightarrow\infty$ and
				\begin{equation}\label{eq2.1}
					u_n=u+\sum_{j=1}^kw^j(\cdot-y_n^j )+o(1)~\text{strongly~in}~H^1(\mathbb{R}^3).
				\end{equation}
				Furthermore, one has
				\begin{equation}\label{2.2}
					\|u_n\|_2^2=\|u\|_2^2+\sum_{j=1}^k\|w^j\|_2^2+o(1)
				\end{equation}
				and
				\begin{equation}\label{2.3}
					J_{V,\lambda}(u_n)= J_{V,\lambda}(u)+\sum_{j=1}^kI_{\lambda}(w^j)+o(1)
				\end{equation}
			\end{lemma}
            \begin{remark}
            In the Splitting Lemma we do not need assumptions $(V_1)$, $(V_1)'$ and $(V_2)$.
            \end{remark}
            \begin{proof}[Proof of the Splitting Lemma]
            The proof of Lemma \ref{lem2.2} is closely similar to \cite[lemma 3.1]{28}, the only differences being the presence of the non-local term and the conditions satisfied by $V$.\\

            Setting $\psi^1_n:=u_n-u$, we have $\psi^1_n\to 0$ weakly in $H^1(\R^3)$ and strongly in $L^p_{loc}(\R^3)$, hence we have
            \begin{equation}
            \label{psi-H1}\|\nabla\psi^1_n\|^2_2=\|\nabla u_n\|^2_2-\|\nabla u\|^2_2+o(1)\qquad\text{as}\, n\to\infty 
            \end{equation}
            \begin{equation}
            \label{psi-L2}\|\psi^1_n\|^2_2=\|u_n\|^2_2-\|u\|^2_2+o(1)\qquad\text{as}\, n\to\infty 
            \end{equation}
            and, thanks to the Br\'ezis-Lieb Theorem (see \cite{BL}),
            \begin{equation}
            \label{brezis-lieb}\|u_n\|^p_p=\|u\|^p_p+\|\psi^1_n\|^p_p+o(1)   \qquad\text{as}\, n\to\infty, 
            \end{equation}
            which yields that
            \begin{equation}
            \label{rel-nabla-I}
            \nabla I_\lambda(\psi^1_n)=\nabla J_{V,\lambda}(u_n)-\nabla J_{V,\lambda}(u)+o(1)\qquad\text{in}\,H^{-1}(\R^3)\,\text{as}\, n\to\infty,
            \end{equation}
            since $\psi^1_n\rightharpoonup 0$ weakly in $H^1(\R^3)$,
            \begin{equation*}
				B(u_n)-B(u)=B(u_n-u)+o(1)\qquad\text{as}\, n\to\infty
		  \end{equation*}
            due to \cite[lemma 2.2]{zz} and the mapping $$v\in H^1(\R^3)\mapsto \int_{\R^3}V\psi v\in\R$$ is in $H^{-1}(\R^3)$, since $V\in L^{\bar{q}}(\R^3)$ for some $\bar{q}\ge \frac{3}{2}$.\\ 
            
            If $\psi^1_n\to 0$ strongly in $H^1(\R^3)$ we are done. Otherwise we have
            \begin{equation}
            \label{Iu-n-unif-pos}
            I_\lambda(\psi^1_n)+\frac{1}{4}B(\psi^1_n)\ge \kappa>0
            \end{equation}
            for some $\kappa>0$ and for $n>0$ large enough. In order to prove (\ref{Iu-n-unif-pos}), we note that (\ref{rel-nabla-I}) gives
            \begin{equation}
            \label{rel-Ipsi}
            I_\lambda(\psi^1_n)+\frac{1}{4}B(\psi^1_n)=\frac{p-2}{2p}\left(\|\psi^1_n\|^2_2+\lambda\|\psi^1_n\|^2_2+\int_{\R^3}V(\psi^1_n)^2dx+B(\psi^1_n)\right)+o(1)=\frac{p-2}{2p}\|\psi^1_n\|^p_p+o(1)\qquad\text{as}\,n\to\infty
            \end{equation}
            As a consequence if we assume that, up to a subsequence, $$I_\lambda(\psi^1_n)+\frac{1}{4}B(\psi^1_n)\to 0\qquad\text{as}\,n\to\infty,$$
            then (\ref{rel-Ipsi}) gives
$$0\le(1-S^{-1}\|V^-\|_{\frac{3}{2}})\|\psi^1_n\|^2_2+\lambda\|\psi^1_n\|^2_2\le \|\psi^1_n\|^2_2+\lambda\|\psi^1_n\|^2_2+\int_{\R^3}V(\psi^1_n)^2dx+B(\psi^1_n)\to 0\qquad\text{as}\,n\to\infty.$$
Here we have used the lower bound
$$\int_{\R^3}V(\psi^1_n)^2dx\ge-\int_{\R^3}V^-(\psi^1_n)^2dx\ge-S^{-1}\|V^-\|_{\frac{3}{2}}\|\psi^1_n\|^2_2.$$
Now we decompose $\R^3$ in the union of countably many closed unit cubes $Q_i$ whose interiors are disjoint and we prove that there exists $\gamma>0$ such that
\begin{equation}
\label{dn-unif-pos}
d_n:=\max_i\|\psi^{1}_n\|_{L^p(Q_i)}\ge \gamma
\end{equation}
In fact, using (\ref{rel-Ipsi}) and the Sobolev embedding $H^1(\R^3)\subset L^p(\R^3)$, we can see that
\begin{equation}\notag
\begin{aligned}
\frac{2p}{p-2}\left(I_\lambda(\psi^1_n)+\frac{1}{4}B(\psi^1_n)\right)+o(1)&=\|\psi^1_n\|^p_p=\sum_i\|\psi^1_n\|^p_{L^p(Q-i)}
\le d_n^{p-2}\sum_i\|\psi^1_n\|^2_{L^p(Q-i)}\\
&\le C_p d_n^{p-2}\sum_i\left(\|\psi^1_n\|^2_{L^2(Q_i)}+\lambda\|\nabla\psi^1_n\|^2_{L^2(Q_i)}\right)\\
&\le \tilde{C}_pd_n^{p-2}\frac{2p}{p-2}\left(I_\lambda(\psi^1_n)+\frac{1}{4}B(\psi^1_n)\right),
\end{aligned}
\end{equation}
which shows that
\begin{equation}\notag
0<\kappa\le \liminf_{n\to\infty}\left(I_\lambda(\psi^1_n)+\frac{1}{4}B(\psi^1_n)\right)\le\limsup_{n\to\infty}(1-\tilde{C}_pd_n^{p-2})\left(I_\lambda(\psi^1_n)+\frac{1}{4}B(\psi^1_n)\right)\le 0
\end{equation}
if $d_n\to 0$, so that (\ref{dn-unif-pos}) is true.\\

Taking $y^1_n$ to be the centre of the hypercube $Q_j$ such that $d_n=\|\psi^1_n\|^p_{L^p(Q_j)}$, we can see that, up to a subsequence, $|y^1_n|\to\infty$. In fact, if $y^1_n$ were bounded, then up to a subsequence $y^1_n\equiv y^1$ would be constant and the corresponding cube $Q_j$ would satisfy $$\|\psi^1_n\|_{L^p(Q_j)}\ge \gamma>0,$$
which is not possible because $\psi^1_n\rightharpoonup 0$ in $H^1(\R^3)$, which implies that $\psi^1_n\to 0$ strongly in $L^p_{loc}(\R^3)$.\\

As a consequence, the sequence $\psi^1_n(\cdotp+y^1_n)$ converges to some solution $w^1$ to
\begin{equation}
\label{lim-eq-R3}
-\Delta w^1+\lambda w^1+V(x)w^1=|w^1|^{p-2}w^1\qquad\text{in}\,\R^3
\end{equation}
weakly in $H^1(\R^3)$ and strongly in $L^p_{loc}(\R^3)$. Applying the above argument to the hypercube centred at the origin, we can see that $w^1\ne 0$. 
$$u_n=u+w^1(\cdotp-y^1_n)+\psi^2_n(\cdotp-y^1_n),\qquad \psi^2_n\rightharpoonup 0\,\text{as}\,n\to\infty\,\text{weakly}\, \text{in}\,H^1(\R^3).$$

Then by induction for $j\ge 2$, we find sequences $y^j_n\in \R^3$ with $|y^j_n|\to\infty$, sequences $\psi^j_n:=\psi^{j-1}_n(\cdotp+y^{j-1}_n)$ and solutions $w^j\in H^1(\R^3)$ such that 
\begin{equation}\label{lim-norms}
\begin{aligned}
\|\nabla\psi^j_n\|^2_2&=\|\nabla\psi^{j-1}_n\|^2_2-\|\nabla w^j\|^2_2+o(1)=\|\nabla u_n\|^2_2-\|\nabla u\|^2_2-\sum_{i=1}^{j-1}\|\nabla w^{i}\|^2_2+o(1)\\
\|\psi^j_n\|^2_2&=\|\psi^{j-1}_n\|^2_2-\|w^j\|^2_2+o(1)=\|u_n\|^2_2-\|u\|^2_2-\sum_{i=1}^{j-1}\|w^{i}\|^2_2+o(1)+o(1)\\
\|\psi^j_n\|^p_p&=\|\psi^{j-1}_n\|^p_p-\|w^j\|^p_p+o(1)=\|u_n\|^p_p-\|u\|^p_p-\sum_{i=1}^{j-1}\|w^{i}\|^p_p+o(1)+o(1)\\
B(\psi^j_n)&=B(\psi^{j-1}_n)-B(w^j)+o(1)=B(u_n)-B(u)-\sum_{i=1}^{j-1}B(w^j)+o(1)
\end{aligned}
\end{equation}
as $n\to\infty$.\\

Now we prove that this iteration stops after a finite number of steps. For this purpose we note that the energy of any solution $w\ne 0$ to problem (\ref{eq1.7}) with $\|w\|_2\le a$ is bounded from above by the energy of the least energy solution $u_a$ provided $a>0$ is small enough. More precisely, the lower bound
\begin{equation}
I_\lambda(w)=I(w)+\lambda\|w\|^2_2>I(w)\ge c_{\|w\|_2}\ge c_a.
\end{equation}
holds. In particular, using that $B(w)>0$, we have
$$I_\lambda(w)+\frac{1}{4}B(w)>c_a$$
for such $w$. Using that $w\in H^1(\R^3)$ is a weak solution to (\ref{lim-eq-R3}), we have
\begin{equation}\notag
\begin{aligned}
&0<c_a<I_\lambda(w)+\frac{1}{4}B(w)=\frac{p-2}{2p}\left(\|\nabla w\|_2^2+\lambda\|w\|^2_2+B(w)\right)\\
&\le c(p)\left(\|\nabla w\|_2^2+\lambda\|w\|^2_2+\|\nabla w\|_2\|w\|^3_2\right)\le 2c(p)\left(\|\nabla w\|_2^2+\lambda\|w\|^2_2+\|w\|^2_{H^1(\R^3)}\|w\|^2_2\right)\\
&\le 2c(p)\max\{1,\lambda\}\|w\|^2_{H^1(\R^3)}(1+\|w\|^2_{H^1(\R^3)}),
\end{aligned}
\end{equation}
which gives $$\|w\|^2_{H^1(\R^3)}\ge \tilde{c}(a,p,\lambda)>0,$$
for any solution $w$ to (\ref{lim-eq-R3}) with $\|w\|_2\le a$. This fact together with (\ref{lim-norms}) shows that the process stops after a certain number $k\in\N$ of iterations.\\

This shows that (\ref{eq2.1}) and (\ref{2.2}) are fulfilled. Now we need to prove that (\ref{2.3}) holds.\\ 

Due to (\ref{lim-norms}), it is enough to show that
            \begin{equation}
            \label{conv-int-Vun}
            \int_{\R^3}V u_n^2 dx\to \int_{\R^3}V u^2 dx\qquad\text{as}\,n\to\infty.
            \end{equation}            
            For this purpose, we observe that for any $\epsilon>0$ there exist $R>0$ and $n_0>0$ such that, for any $n\ge n_0$ and $1\le j\le k$ we have
           \begin{equation}
           \begin{aligned}
           \int_{\R^3}V w^j_n(x-y^j_n)^2 dx&=
           \int_{B_R(0)} V w^j_n(x-y^j_n)^2 dx+\int_{\R^3\setminus B_R(0)} V w^j_n(x-y^j_n)^2dx\\
           &\le \|V\|_{\bar{q}}\left(\int_{\R^3\setminus B_{|y^j_n|-R}(0)}|w^j|^{\frac{2\bar{q}}{\bar{q}-1}}dx\right)^{\frac{\bar{q}-1}{\bar{q}}}+\epsilon\|u^j\|^2_2<c\epsilon
           \end{aligned}
           \end{equation}
           by the embedding $H^1(B_R(0))\subset L^{\frac{2\bar{q}}{\bar{q}-1}}(B_R(0))$ and the fact that $\lim_{|x|\to\infty}V(x)=0$ and $|y^j_n|\to\infty$ as $n\to\infty$. A similar argument shows that
           $$\int_{\R^3}Vw^j(x-y^j_n)u dx\to 0\qquad\text{as}\,n\to\infty.$$
            \end{proof}
Now we can conclude the proof of Theorem \ref{theorem1.2}. 
We will see that this is the only point in which we use $(V_2)$. The argument is similar to the one used in \cite{BQZ} for the poofs of Theorems $1.6,\,1.8$ and $1.9$. Since we consider the nonlocal problem,  we must take care of the decay estimate of the nonlocal term $\phi_u$.
			
			\begin{proof}[\textbf{Proof of Theorem \ref{theorem1.2}}]
				Let $\Omega:=B_1(0)$ be the unit ball. For any $a\in(0,a^*)$ and $r\geq r_a$, we consider the solution $(\lambda_r,u_r)\in\R\times H^1(\Omega_r)$ to Problem \eqref{eq-Omega_r} constructed in Theorem \ref{th-bd-V_1}. From Propositions \ref{prop-existence-sol} and  Lemma \ref{lemma-PS}, there exists a sequence $r_n\to\infty$ such that the sequence $\{u_n\}:=\{u_{r_n}\}_n\subset H^1(\R^3)$ is a bounded Palais-Smale sequence of $E_{\infty,1}$ which is weakly converging to a solution $u\in H^1(\R^3)$ to 
				$$-\Delta u+\lambda u+V(x)u+\phi_u u=|u|^{p-2}u\quad\text{in}\quad\R^3,$$
				with $u\ge 0$ and $\lambda>0$. It remains to prove that $u\in \mathcal{S}_a$.\\ 
				
				By Splitting Lemma \ref{lem2.2}, we have
				$$u_n=u+\sum_{j=1}^k w^j(\cdotp-y^j_n)+o(1)\quad\text{strongly}\quad\text{in}\quad H^1(\R^3)$$
				where $k\ge 1$ is an integer and $w^1,\dots,w^k$ are solutions to the (\ref{lim-eq}).\\
                
                Moreover, setting $\tilde{a}:=\|u\|_2$ and $a_j:=\|w^j\|_2$, for $1\le j\le k$, we have
				$$a^2=\tilde{a}^2+\sum_{j=1}^k a_j^2+o(1),\qquad J_V(u_n)=J_V(u)+\sum_{j=1}^k J_\infty(w^j)+o(1).$$
                If $k=0$, then $u_n\to u$ strongly in $H^1(\R^3)$ and the proof is done. If we assume by contradiction that $k\ge 1$, then, up to a subsequence, there exists $n_0>0$ and $i\in\{1,\dots,k\}$ such that $$|y^{i}_n|=\min\{|y^{j}_n|:\,1\le j\le k\},\,\forall\,n\ge n_0,\qquad \lim_{n\to\infty}\frac{|y^{i}_n-y^{j}_n|}{|y^{i}_n|}= d_j\in[0,\infty],\,\forall\,1\le j\le k.$$ 
                Assume without loss of generality that $i=1$. We take $0<\delta<\rho$, where $\rho>0$ is defined in $(V_2)$, and  consider the annulus $$A_n:=B_{\frac{3\delta}{2}|y^1_n|}(y^1_n)~\backslash~ B_{\frac{\delta}{2}|y^1_n|}(y^1_n).$$
                We prove that, if $\delta$ is small enough and $n_0$ is large enough, then 
                \begin{equation}\label{y^j-far}
            {\rm dist}(y^{j}_n,A_n)>\frac{\delta}{4}, \qquad\forall\,n\ge n_0,\,1\le j\le k
                \end{equation}
                In fact, setting $K:=\{j\in\{1,\dots,k\}:\, d_j>0\}$, we can see that, if $j\in\{1,\dots,k\}\setminus K$, then $d_j=0$ and hence, for $n$ large, $y^{j}_n\in B_{\frac{\delta}{4}|y^{1}_n|}(y^{1}_n)$, so that (\ref{y^j-far}) is true. On the other hand, if $j\in K$, then $$|y^{j}_n-y^{1}_n|\ge \frac{1}{2}|y^{1}_n|d_j\ge\frac{1}{2}\min\{d_\ell:\,1\le \ell\le k\}|y^{1}_n|>\frac{\delta}{4}|y^{1}_n|$$
                provided $n$ is large enough and $\delta$ is small enough, hence (\ref{y^j-far}) is satisfied as well.\\

                In particular, (\ref{y^j-far}) yields that
                \begin{equation}\label{un-L2-An}
                    \|u_n\|_{L^2(A_n)}\to 0\qquad\text{as}\,n\to\infty
                \end{equation}
                In fact
                \begin{equation}\notag
                \begin{aligned}
                &\|u_n\|_{L^2(A_n)}\le \|u\|_{L^2(A_n)}+\sum_{j=1}^k\|w^j(\cdotp-y^{j}_n)\|_{L^2(A_n)}+o(1)\\
                &\le \|u\|_{L^2(\R^3\setminus B_{\frac{\delta|y^1_n|}{2}}(0))}+\sum_{j=1}^k\|w^j\|_{L^2(\R^3\setminus B_{\frac{\delta|y^1_n|}{4}}(0)))}+o(1)=o(1)
                \end{aligned}
                \end{equation}
                as $n\to\infty$, thanks to (\ref{y^j-far}).

				Now we prove that, for any $1\le j\le k$, the sequence $\{y^{j}_n\}$ satisfies 
				\begin{equation}
					\label{y-n-far-boundary}
					{\rm dist}(y^j_n,\R^3\setminus B_{r_n}(0))\to\infty\quad\text{as}\quad n\to\infty.
				\end{equation}
				Assume by contradiction that (\ref{y-n-far-boundary}) does not hold and there is $R>0$ such that, up to a subsequence,  $$\sup_{n}{\rm dist}(y^j_n,\R^3\setminus B_{r_n}(0))<R,$$ 
                for some $j\in\{1,\dots,k\}$. Assume that, up to a subsequence,  $\frac{y^j_n}{|y^j_n|}\to e\in S^{2}$, where $S^2\subset \R^3$ denotes the unit $2$-sphere. Then, setting $\Sigma_n:=B_{2R}(y^1_n)\setminus B_{r_n}(0)$, we have as $n\rightarrow\infty$,
				\begin{equation}\notag
					\begin{aligned}
						0&=\int_{\Sigma_n}|u_n|^2 dx=\int_{\Sigma_n}|u(x)+\sum_{\ell=1}^kw^\ell(x-y^j_n)+o(1)|^2dx=\int_{\Sigma_n}|w^j(x-y^j_n)|^2 dx+o(1)\\
						&\ge \int_{B_{R/4}(y^j_n+\frac{3}{2}R\frac{y^j_n}{|y^1_n|})}|w^j(x-y^j_n)|^2 dx+o(1)\ge\int_{B_{R/4}(Re)}|w^j|^2 dx>0,
					\end{aligned}    
				\end{equation}
				which is impossible and (\ref{y-n-far-boundary}) is proved.\\
			
				Extending $u_n$ to the whole $\R^3$ by setting it to be $0$ outside $B_{r_n}(0)$, we can get a non-negative subsolution to (\ref{eq-R3}) in $\R^3$, which we still denote by $u_n$. In other words, the differential inequality
				\begin{equation}
					\label{u_n-subsol}
					-\Delta u_n+V(x)u_n+\tilde{\lambda} u_n+(|x|^{-1}\ast |u_n|^2)u_n\le|u_n|^{p-2}u_n\quad\text{in}\quad\R^3
				\end{equation}
				is fulfilled for large $r$, where $\tilde{\lambda}:=\frac{1}{2}\liminf\limits_{n\rightarrow\infty}\lambda_{r_n}>0$ due to Proposition \ref{prop-existence-sol}. For $m,\,n\ge 1$, setting
				$$R_{m,n}:=\overline{B_{\frac{3\delta}{2}|y^1_n|-m}(y^1_n)}~\backslash ~B_{\frac{\delta}{2}|y^1_n|+m}(y^1_n),$$
				then by \cite[Theorem 8.17]{Tru} and (\ref{un-L2-An}), we can see that, for $n$ large enough, 
				\begin{equation}
					\label{est-u-annulus}
					\|u_n\|^{p-2}_{L^\infty(R_{1,n})}\le c\|u_n\|^{p-2}_{L^2(A_n)}<\frac{\tilde{\lambda}}{4},
				\end{equation}
				for some constant $c>0$.\\ 
				
				We take, for $m\ge 1$ a cutoff function $\xi_m\in C^\infty(\R)$ such that $0\le \xi_m(t)\le 1$ and $|\xi_m'(t)|\le 4$, for any $m\ge 1$ and $t\in\R$, and
				$$\xi_m(t)=
				\begin{cases}
					1\quad\text{if}\, \frac{\delta}{2}|y^1_n|+m<t< \frac{3\delta}{2}|y^1_n|-m,\\
					0\quad\text{if}\,t< \frac{\delta}{2}|y^1_n|+m-1\,\text{or}\,t> \frac{3\delta}{2}|y^1_n|-m+1
				\end{cases}
				$$
				Setting $\psi_m(x):=\xi_m(|x-y^1_n|)$ and testing inequality (\ref{u_n-subsol}) with $\psi_m^2 u_n\ge 0$, we have
				\begin{equation}\label{upper-est}
					\begin{aligned}
						&\int_{R_{m-1,n}}|\nabla u_n|^2\psi_m^2 dx+\int_{R_{m-1,n}}(V(x)+\tilde{\lambda} )u_n^2\psi_m^2 dx-\int_{R_{m-1,n}}|u_n|^p\psi_m^2 dx\\
						&+\int_{R_{m-1,n}}\phi_{u_n}u_n^2\psi_m^2 dx\le-2\int_{R_{m-1,n}}u_n\psi_n\nabla u_n\cdotp\nabla \psi_m dx\le 8\int_{R_{m-1,n}\setminus R_{m,n}}\psi_n|\nabla u_n||u_n|dx\\
						&\le 4\left(\int_{R_{m-1,n}\setminus R_{m,n}}(|\nabla u_n|^2+u_n^2) dx\right)=4(b_m-b_{m-1}),
					\end{aligned}
				\end{equation}
				where we have set $$b_m:=\int_{R_{m,n}}(|\nabla u_n|^2+u_n^2) dx.$$
				On the other hand, using (\ref{est-u-annulus}) and the fact that $V\ge 0$, it is possible to see that
				\begin{equation}\label{lower-est}
					\begin{aligned}
						&\int_{R_{m-1,n}}|\nabla u_n|^2\psi_m^2 dx+\int_{R_{m-1,n}}(V(x)+\tilde{\lambda} )u_n^2\psi_m^2 dx-\int_{R_{m-1,n}}|u_n|^p\psi_m^2 dx\\
						&+\int_{R_{m-1,n}}\phi_{u_n}u_n^2\psi_m^2 dx\ge \int_{R_{m-1,n}}|\nabla u_n|^2\psi_m^2 dx+ \frac{\lambda}{4}\int_{R_{m-1,n}}u_n^2\psi_m^2 dx\\
						&\ge\min\left\{1,\frac{\lambda}{4}\right\}\int_{R_{m-1,n}}(|\nabla u_n|^2+u_n^2)\psi_m^2 dx
						\ge\min\left\{1,\frac{\lambda}{4}\right\} b_m. 
					\end{aligned}
				\end{equation}
				As a consequence, setting $\kappa:=4\max\left\{1,\frac{4}{\lambda}\right\}>0$ and $\vartheta:=\frac{\kappa}{\kappa+1}\in(0,1)$, estimates (\ref{upper-est}) and (\ref{lower-est}) give
				$$b_m\le\vartheta b_{m-1}\le\vartheta^m\max_n\|u_n\|_{H^1(\R^3)}\le K_a e^{m\log\vartheta},$$
				for some constant $K_a>0$ depending on $a$ only, since $\{u_n\}$ is bounded in $H^1(\R^3)$.\\
				
                Taking
                $m:=\left[\frac{\delta|y^1_n|}{4}\right]-1$, 
                we have
				$$\int_{R_{m,n}}(|\nabla u_n|^2+u_n^2)dx\le ce^{m\log\vartheta}\le ce^{\frac{\delta|y^1_n|}{4}\log\vartheta},$$
				so that, by \cite[Theorem 8.17]{Tru} again and the elliptic estimates up to the boundary given by   \cite[Theorem 9.13]{Tru}, one has
				\begin{equation}
					\label{est-u-n-p-omega-n}
					|\nabla u_n(x)|^2+u_n^2(x)\le c e^{-\tilde{c}|y^1_n|}\quad\forall\quad x\in B_{\frac{5\delta}{4}|y^1_n|-1}(y^1_n)\setminus B_{\frac{3\delta}{4}|y^1_n|+1}(y^1_n),
				\end{equation}
				for some $c,\,\tilde{c}>0$.\\
				
				Arguing as above, it is possible to show that, changing if necessary the values of $c$ and $\tilde{c}$,
				\begin{equation}
					\label{est-u-n-p-B-rn}
					u^2(x)+|\nabla u_n(x)|^2\le c e^{-\tilde{c}r_n}\quad\forall\quad x\in \bar{B}_{r_n}(0)\setminus B_{\frac{3r_n}{4}+1}(0).
				\end{equation}
				For this purpose, it is enough to set
				$$R_{m,n}:=B_{r_n+m}(0)\setminus B_{r_n-m}(0),$$
				take $\xi_m\in C^\infty_c(\R)$ such that
				$$\xi_m(t)=
				\begin{cases}
					1\qquad\text{if}\,r_n-m<t<r_n+m,\\
					0\qquad\text{if}\,t<r_n-m-1\,\text{or}\,t>r_n+m+1
				\end{cases}
				$$
				and observe that 
				\begin{equation}
					\label{est-u-annulus'}
					\|u_n\|^{p-2}_{L^\infty(R_{m,n})}\le c\|u_n\|^{p-2}_{L^2(R_{1,n})}\le\frac{\tilde{\lambda}}{4}\quad\forall\, m\ge 2,
				\end{equation}
				for $n$ large enough, since ${\rm dist}(y^1_n,\R^3 \setminus B_{r_n}(0))\to\infty$ as $n\to\infty$ (see (\ref{y-n-far-boundary})). Setting once again $\psi_m(x):=\xi_m(|x-y^1_n|)$ and testing the differential inequality (\ref{u_n-subsol}) with $u_n\psi^2_m$, we have
				$$\int_{B_{r_n}(0)\setminus B_{r_n-1}(0)}(|\nabla u_n|^2+u_n^2)dx\le c e^{-\tilde{c}r_n},$$
				for some constants $c,\,\tilde{c}>0$. Using once again \cite[ Theorems 8.17 and 9.13]{Tru}, we conclude that (\ref{est-u-n-p-B-rn}) is true.\\
				
				As we observed in Remark \ref{rem-test-function}, we can test equation (\ref{eq-Omega_r}) with $\nabla u_n\cdotp y^1_n$, by Divergence Theorem, it yields
				\begin{equation}
					\label{magic-T-formula}
					\begin{aligned}
						\frac{1}{2}\int_{\omega_n}u_n^2(\nabla V(x)\cdotp y^1_n+\nabla\phi_{u_n}\cdotp y^1_n)dx&=\int_{\partial\omega_n}(\nabla u_n\cdotp\nu)(\nabla u_n\cdotp y^1_n)d\sigma-\frac{1}{2}\int_{\partial\omega_n}y^1_n\cdotp\nu |\nabla u_n|^2 d\sigma\\
						&+\int_{\partial\omega_n}(y^1_n\cdotp\nu)\left(\frac{\lambda}{2}u_n^2+\frac{V(x)}{2}u_n^2+\frac{1}{2}\phi_{u_n}u_n^2-\frac{u_n^p}{p}\right)d\sigma
					\end{aligned}    
				\end{equation}
				where we have set $\omega_n:=B_{\delta|y^1_n|}(y^1_n)\cap B_{r_n}(0)$ and $\nu$ is its exterior normal vector, which is defined almost everywhere on $\partial\omega_n$.\\
				
				Multiplying (\ref{magic-T-formula}) by $|y^1_n|^\alpha$, we can see that there exists $K>0$ and $n_K>0$ such that, for any $n>n_K$ we have
				\begin{equation}\label{one-hand}
						|y^1_n|^\alpha\int_{\omega_n}u_n^2 \nabla V(x)\cdotp y^1_n dx\le   \sup_{x\in B_{\delta|y^1_n|}(y^1_n)}\left(|y^1_n|^\alpha\nabla V(x)\cdotp y^1_n\right)\int_{\omega_n}|u_n|^2 dx\\
						\le-K\int_{\omega_n}|u_n|^2 dx
				\end{equation}
				due to assumption $(V_2)$ and the fact that $$\liminf_{n\to\infty}\int_{\omega_n}|u_n|^2 dx\ge \int_{\R^3}|w^1|^2 dx=a_1^2>0.$$
                
				Using that $u_n=0$ on $\partial B_{r_n}(0)$, 
                 the exponential decay of $u_n$ and $\nabla u_n$ on $\partial \omega_n$ given by (\ref{est-u-annulus}) and (\ref{est-u-annulus'}) and the uniform bound of $\phi_{u_n}$ in $W^{1,\infty}(\R^3)$ given by Lemma \ref{lemma-est-nabla-phi}, we can see that the right hand side of (\ref{magic-T-formula}) fulfills
				\begin{equation}\notag
					\begin{aligned}
						&|y^1_n|^\alpha \bigg(\int_{\partial\omega_n}(\nabla u_n\cdotp\nu)(\nabla u_n\cdotp y^1_n)d\sigma-\frac{1}{2}\int_{\partial\omega_n}y^1_n\cdotp\nu |\nabla u_n|^2 d\sigma\\
						&+\int_{\partial\omega_n}(y^1_n\cdotp\nu)\left(\frac{\lambda}{2}u_n^2+\frac{V(x)}{2}u_n^2+\frac{1}{2}\phi_{u_n}u_n^2-\frac{u_n^p}{p}\right)d\sigma\bigg)\to 0
					\end{aligned}
				\end{equation}
				as $n\to\infty$. 
                
                Therefore, in order to have a contradiction with \eqref{one-hand} it is enough to show that 
                \begin{equation}
                \label{est-nonloc-term}
                |y^1_n|^\alpha\int_{\omega_n}u_n^2\nabla\phi_{u_n}\cdotp y^1_n dx\to 0\qquad\text{as}\, n\to\infty.    
                \end{equation}
                Using that $u_n=0$ in $\R^3\setminus B_{r_n}(0)$ and $\nabla\phi_{u_n}\cdotp y^1_n\in H^1(B_{\delta|y^1_n|}(y^1_n))$, due to Lemma \ref{lemma-est-nabla-phi}, we can test the equation
                $$-\Delta\phi_{u_n} =u_n^2\qquad\text{in}\,\R^3,$$
                with $\nabla\phi_{u_n}\cdotp y^1_n$ in $B_{\delta|y^1_n|}(y^1_n)$, which gives
                \begin{equation}\notag
                \begin{aligned}\int_{\omega_n}u_n^2\nabla\phi_{u_n}\cdotp y^1_n dx&=\int_{B_{\delta|y^1_n|}(y^1_n)}u_n^2\nabla\phi_{u_n}\cdotp y^1_n dx=-\int_{B_{\delta|y^1_n|}(y^1_n)}\Delta\phi_{u_n}\nabla\phi_{u_n}\cdotp y^1_n dx\\
                &=\int_{\partial B_{\delta|y^1_n|}(y^1_n)}\left(\frac{1}{2}\nu\cdotp y^1_n|\nabla \phi_{u_n}|^2-(\nabla\phi_{u_n}\cdotp\nu)(\nabla\phi_{u_n}\cdotp y^1_n)\right)d\sigma.
                \end{aligned}
                \end{equation}
                Differentiating under the integral sign and using estimate (\ref{est-u-n-p-omega-n}) about the behaviour of $u$ near $\partial B_{\delta|y^1_n|}(y^1_n)$ we have 
                \begin{equation}\notag
                \begin{aligned}
                |\nabla\phi_{u_n}(x)|&\le\int_{\R^3}\frac{u_n^2(z)}{|x-z|^2}dz=\int_{B_{\frac{\delta|y^1_n|}{4}}(x)}\frac{u_n^2(z)}{|x-z|^2}dz+\int_{\R^3\setminus B_{\frac{\delta|y^1_n|}{4}}(x)}\frac{u_n^2(z)}{|x-z|^2}dz\\
                &\le ce^{c|y^1_n|}\int_{B_{\frac{\delta|y^1_n|}{4}(0)}}\frac{dz}{|z|^2}+\frac{16}{\delta^2|y^1_n|^2}\int_{\R^3}u_n^2 dx\le \frac{c_\delta}{|y^1_n|^2}\qquad\forall\,x\in\partial B_{\delta|y^1_n|}(y^1_n),
                \end{aligned}                    
                \end{equation}
                for some constant $c_\delta>0$ depending on $\delta$. As a result we have
                $$\left|\int_{\omega_n}u_n^2\nabla\phi_{u_n}\cdotp y^1_n dx\right|=\left|\int_{\partial B_{\delta|y^1_n|}(y^1_n)}\left(\frac{1}{2}\nu\cdotp y^1_n|\nabla \phi_{u_n}|^2-(\nabla\phi_{u_n}\cdotp\nu)(\nabla\phi_{u_n}\cdotp y^1_n)\right)d\sigma\right|\le\frac{c_\delta|y^1_n|}{|y^1_n|^4}|\partial B_{r_n}(0)|=\frac{\tilde{c}_\delta}{|y^1_n|}$$
                for $n>0$ large enough, which shows that (\ref{est-nonloc-term}) is true.
                \vskip 0.1in
                Finally, we have $k=0$ and $u_n\to u$ strongly in $H^1(\R^3)$, which yields that $$c_a\le m_{r_n,1}(a)= E_{r_n,1}(u_n)\to J_V(u),$$
                so that $J_V(u)\ge c_a$ and the proof of Theorem \ref{theorem1.2} is concluded.
                \end{proof}

\section{The proof of Theorem \ref{theorem1.3}: possibly unbounded non-negative potential}\label{sec_thm1.3}
The proof of Theorem \ref{theorem1.3} follows the outlines of the one of Theorem \ref{theorem1.2}. In this Section we will assume that $(V_1)'$ holds and discuss the differences.\\

First we prove Lemma \ref{lemma-MP-geom} in case $(V_1)'$ holds.
\begin{proof}[Proof of Lemma \ref{lemma-MP-geom} if $(V_1)'$ holds] 
The proof is similar to the case in which $(V_1)$ holds. We will only discuss the differences.
\vskip0.1in
            \begin{itemize}
            \item \textit{Construction of $u^1$}.
            \vskip0.1in
            For $a\in(0,a_0)$ and $s\in[\frac{1}{2},1]$, setting $$f_{a,s}(t):=\frac{t^2}{2}\|\nabla u_a\|^2_2+\frac{t^{\frac{3}{q}}}{2}\|V\|_q\|u_a\|^2_{\frac{2q}{q-1}}+\frac{t}{4}B(u_a)-\frac{s}{p}\|u_a\|^p_p t^{\frac{3(p-2)}{2}},$$
            relation (\ref{main-est-E-rs}) becomes
            \begin{equation}
            E_{\infty,\frac{1}{2}}(\gamma_\infty(t))\le \bar{E}_{\infty,\frac{1}{2}}(\gamma_\infty(t))+\frac{t^{\frac{3}{q}}}{2}\|V\|_q\|u_a\|^2_{\frac{2q}{q-1}}\\
            \le f_{a,\frac{1}{2}}(t),    
            \end{equation}
            for any $a\in(0,a_0)$ and $t>0$.\\

            \item \textit{Construction of $u^0$.}\\
            
            Here the only difference is that (\ref{est-E-u0}) becomes
            $$\lim_{r\to\infty}E_{r,\frac{1}{2}}(\gamma_{r}(t))\le f_{a,\frac{1}{2}}(t)<c_a\qquad\forall\,a\in(0,a_0),\,t\in(0,t_{2,a}],$$
            for some $t_{2,a}>0$.
\vskip0.1in
            \item \textit{Lower estimate of the mountain pass level: for given  $0<a<a_0$, $m_{r,s}(a)\ge c_a$, for any $s\in[\frac{1}{2},1]$ and  $r\ge r_a$, if $r_a>0$ is large enough.}
            \vskip0.1in
            Here nothing changes due to the fact $V\ge 0$.
            \vskip0.1in
            \item \textit{Upper estimate of the mountain-pass level: for given $0<a<a_0$ there exists $\tilde{c}_a>c_a$ such that $m_{r,s}(a)\le\tilde{c}_a$, for any $s\in[\frac{1}{2},1]$ and  $r\ge r_a$, if $r_a>0$ is large enough.}
            \vskip0.1in
            First we recall that the following relations hold
\begin{equation}
\label{energy}
\|\nabla u_a\|_2^2+\frac{1}{2}B(u_a)-\frac{2}{p}\|u_a\|^p_p=2 m_a,
\end{equation}
\begin{equation}
\label{nabla}
\|\nabla u_a\|_2^2+B(u_a)-\|u_a\|^p_p=-\lambda_a a^2,
\end{equation}
\begin{equation}
\label{Pohozaev}
\|\nabla u_a\|_2^2+\frac{1}{4}B(u_a)-\frac{3(p-2)}{2p}\|u_a\|^p_p=0
\end{equation}
Using (\ref{nabla}) and the Pohozaev identity (\ref{Pohozaev}), we have $$m_a=\frac{3p-10}{4p}\|u_a\|_p^p+\frac{1}{8}B(u_a).$$
            Differentiating and using once again (\ref{Pohozaev}) we can see that
            \begin{equation}\notag
            \begin{aligned}
            f'_{a,s}(t)&=t \|\nabla u_a\|^2_2+\frac{3}{2q}t^{\frac{3}{q}-1}\|V\|_q\|u_a\|^2_{\frac{2q}{q-1}}+\frac{B(u_a)}{4}-s\frac{3(p-2)}{2p}\|u_a\|^p_pt^{\frac{3p-8}{2}}\\
            &\ge t^{\frac{3p-8}{2}}\left(\|\nabla u_a\|^2_2+\frac{B(u_a)}{4}-\frac{3(p-2)}{2p}\|u_a\|^p_p\right)+\frac{3}{2q}t^{\frac{3}{q}-1}\|V\|_q\|u_a\|^2_{\frac{2q}{q-1}}\\
            &=\frac{3}{2q}t^{\frac{3}{q}-1}\|V\|_q\|u_a\|^2_{\frac{2q}{q-1}}>0
            \end{aligned}
            \end{equation}
            for any $t\in(0,1]$ and $s\in[\frac{1}{2},1]$, which yields that for any $s\in[\frac{1}{2},1]$ there exists $t_s>1$ such that
            \begin{equation}\notag
                f_{a,s}(t_s)=\max_{t\in(0,1)}f_{a,s}(t),\qquad f'_{a,s}(t_s)=0.
            \end{equation}
            Using the fact $f_{a,s}'(t_s)=0$ and $t_s>1$, we have
            \begin{equation}\label{est-max-f}
            \begin{aligned}
            f_{a,s}(t_s)
            &< t_s^{\frac{3(p-2)}{2}}\left(\frac{3p-10}{4p}\|u_a\|_p^p+\frac{1}{8}B(u_a)\right)+\frac{1}{2}\|V\|_q \|u_a\|^2_{\frac{2q}{q-1}}t_s^{\frac{3}{q}}\left(1-\frac{3}{2q}\right)\\
            &=c_a t_s^{\frac{3(p-2)}{2}}+\frac{1}{2}\|V\|_q \|u_a\|^2_{\frac{2q}{q-1}}t_s^{\frac{3}{q}}\left(1-\frac{3}{2q}\right)\qquad\forall\, s\in[\frac{1}{2},1].
            \end{aligned}
            \end{equation}
            Using once again the Pohozaev identity (\ref{Pohozaev}) and the fact that $t_s>1$, we can see that
            \begin{equation}\notag
            \begin{aligned}
            s\frac{3(p-2)}{2p}\|u_a\|_p^p t_s^{\frac{3p-8}{2}}&=\frac{B(u_a)}{4}+t_s\|\nabla u_a\|^2_2+\frac{3}{2q}t_s^{\frac{3}{q}-1}\|V\|_q\|u_a\|^2_{\frac{2q}{q-1}}\\
            &<t_s\left(\frac{B(u_a)}{4}+\|\nabla u_a\|^2_2+\frac{3}{2q}\|V\|_q\|u_a\|^2_{\frac{2q}{q-1}}\right)\\
            &=t_s\left(\frac{3(p-2)}{2p}\|u_a\|_p^p+\frac{3}{2q}\|V\|_q\|u_a\|^2_{\frac{2q}{q-1}}\right),
            \end{aligned}
            \end{equation}
            which yields that
            $$t_s^{\frac{3p-10}{2}}<\frac{1}{s}\left(1+\frac{p}{q(p-2)}\frac{\|V\|_q\|u_a\|^2_{\frac{2q}{q-1}}}{\|u_a\|^p_p}\right)\le\tilde{t}_a,$$
            for any $s\in[\frac{1}{2},1]$, for some $\tilde{t}_a$ independent of $s$. In conclusion, thanks to (\ref{est-max-f}), the statement holds true. 
            \end{itemize}
       \vskip0.1in
            Now we set
            $$\vartheta(t):=\left(1+\frac{p C_q^2}{q(p-2)}\max\left\{1,\frac{3p-8}{p}\right\}t\right)^{\frac{3(p-2)}{3p-10}}\left(1+t\left(1-\frac{3}{2q}\right)\frac{3C_q^2(p-2)}{3p-10}\right)-1\qquad\forall\, t\ge 0$$
            and we prove that, for given $0<a<a_0$, we have 
            \begin{equation}
            \label{bound-m1a-V'}
            m_{r,1}(a)<(1+\theta)c_a\qquad\forall\, r\ge r_a
            \end{equation}
            provided $r_a>0$ is large enough, where $\theta:=\vartheta(\|V\|_q)>0$.\\

            In order to do so we show a lower bound for $\frac{\|u_a\|^p_p}{\|u_a\|^2_{\frac{2q}{q-1}}}$ which is independent of $a$. Solving the system given by equations (\ref{energy}), (\ref{nabla}) and (\ref{Pohozaev}), we have the explicit expressions
\begin{equation}
\begin{aligned}
\|\nabla u_a\|_2^2&=\frac{3p-8}{4(p-3)}\lambda_a a^2+\frac{5p-12}{2(p-3)}c_a,\\
B(u_a)&=\frac{3p-10}{2(p-3)}\lambda_a a^2+\frac{6-p}{p-3}c_a,\\
\|u_a\|_p^p&=\frac{p}{4(p-3)}\lambda_a a^2+\frac{3p}{2(p-3)}c_a.
\end{aligned}
\end{equation}
As a consequence, denoting the best constant in the Sobolev embedding $H^1(\R^3)\subset L^{\frac{2q}{q-1}}(\R^3)$ by $C_q$, we have the lower bound
\begin{equation}\label{bound-Lp}
\begin{aligned}
&\frac{\|u_a\|^p_p}{\|u_a\|^2_{\frac{2q}{q-1}}}\ge\frac{1}{C_q^2}\frac{p(\lambda_a a^2+6 c_a)}{(3p-8)\lambda_a a^2+(10p-24)c_a+4(p-2)a^2}\\
&>\frac{1}{C_q^2}\frac{p(\lambda_a a^2+6 c_a)}{(3p-8)\lambda_a a^2+6pc_a+4(p-2)a^2}>\frac{1}{C_q^2}\min\left\{\frac{p}{3p-8},1\right\}=:\Lambda_{p,q}, 
\end{aligned}
\end{equation}
provided $a>0$ is small enough, since $c_a\to\infty$ as $a\to 0^+$.\\

Moreover, by (\ref{Pohozaev}) and (\ref{energy}), we have $$\|u_a\|^2_{\frac{2q}{q-1}}\le C^2_q\|\nabla u_a\|^2_2\le\frac{6C^2_q(p-2)}{3p-10}c_a.$$ 
As a consequence 
\begin{equation}\notag
\begin{aligned}
&\lim_{r\to\infty}m_{r,1}(a)=c_a\le f_{a,1}(t_1)< t_1^{\frac{3(p-2)}{2}}c_a+\frac{1}{2}\|V\|_q \|u_a\|^2_{\frac{2q}{q-1}}t_1^{\frac{3}{q}}\left(1-\frac{3}{2q}\right)\\
&\le t_1^{\frac{3(p-2)}{2}}c_a\left(1+\|V\|_q\left(1-\frac{3}{2q}\right)\frac{3C_q^2(p-2)}{3p-10}\right)\\
&<\left(1+\frac{p}{q(p-2)}\frac{\|V\|_q\|u_a\|^2_{\frac{2q}{q-1}}}{\|u_a\|^p_p}\right)^{\frac{3(p-2)}{3p-10}}\left(1+\|V\|_q\left(1-\frac{3}{2q}\right)\frac{3C_q^2(p-2)}{3p-10}\right)c_a\\
&=(1+\vartheta(\|V\|_q))c_a,
\end{aligned}
\end{equation}
that is (\ref{bound-m1a-V'}) is fulfilled. We note that we do not need any smallness assumption about $\|V\|_q$ and $\|W\|_q$, provided we take the suitable value of $\theta$ in the upper bound (\ref{bound-m1a-V'}).
\end{proof}
The proof of the Pohozaev identity is similar. The only difference is that we need to be careful to justify the differentiation under the integral sign. We stress that, in case $(V_1)'$ holds, we can prove the Pohozaev identity in case $r<\infty$ only, that is in bounded domains only, which is precisely what we need. Moreover, we need the property $\Omega_r\subset\Omega_{r'}$ if $r>r'$, which holds true if $0\in\Omega$.
\begin{proof}[Proof of Lemma \ref{lemma-pohozaev} if $(V_1)'$ holds in case $r\in(r_a,\infty)$ and $0\in\Omega$]
We need to prove that, if $V\in W^{1,3/2}_{loc}(\R^3)$, then
\begin{equation}
    \frac{d}{dt}\left(\int_{\Omega_r}V\left(\frac{x}{t}\right)u(x)^2 dx\right)=\frac{1}{t}\int_{\Omega_r}W(\frac{x}{t})u^2(x)dx\qquad\forall\,t>0,\,u\in H^1(\R^3),
\end{equation}
where $W(x):=\nabla V(x)\cdotp x$. For this purpose we note that, due to the Friedrichs Theorem (see \cite{Brezis}), for any $\sigma>0$ and $r>r_a$ there exists a sequence of functions $V_n\in C^\infty_c(\R^3)$ such that $V_n\to V$ in $L^{3/2}(\Omega_{\frac{r+1}{\sigma}})$ and $\nabla V_n\to\nabla V$ in $L^{3/2}(\Omega_{r/\sigma})$. As a consequence, the sequence $W_n(x):=\nabla V_n(x)\cdotp x$ fulfills $W_n\to W$ in $L^{3/2}(\Omega_{r/\sigma})$ as well. 
Hence for any  $\varepsilon>0$ there exists $n_0(\varepsilon,\sigma)>0$ such that, for any $n\ge n_0(\varepsilon,\sigma)$ and $t\in(\sigma,\frac{1}{\sigma})$, we have
\begin{equation}\notag
\begin{aligned}
&\left|\int_{\Omega_r}\left(V_n\left(\frac{x}{t}\right)-V\left(\frac{x}{t}\right)\right)u^2(x)dx\right|
\le t^2\|V_n-V\|_{L^{3/2}(\Omega_{r/t})}\|u\|^2_6\le\sigma^2\|V_n-V\|_{L^{3/2}(\Omega_{r/\sigma})}\|u\|^2_6<c\varepsilon.
\end{aligned}    
\end{equation}
In other words, we have proved that
$$F_n(t):=\int_{\Omega_r}V_n\left(\frac{x}{t}\right)u^2(x)dx\to F(t):=\int_{\Omega_r}V\left(\frac{x}{t}\right)u^2(x)dx$$
uniformly on compact subsets of $(0,\infty)$. Similarly, it is possible to see that
$$F'_n(t):=-\frac{1}{t}\int_{\Omega_r}W_n\left(\frac{x}{t}\right)u^2(x)dx\to G(t):=-\frac{1}{t}\int_{\Omega_r}W\left(\frac{x}{t}\right)u^2(x)dx$$
uniformly on compact subsets of $(0,\infty)$. As a consequence, we get $F\in C^1(0,\infty)$ and $F'=G$.
\end{proof}
Lemma \ref{lemma-unif-bound-s}, Remark \ref{rem-unif-est-u-rs} and Proposition \ref{prop-sol-Omegar-s} can be summarized as follows.
\begin{proposition}\label{lemma-unif-bound'}
Assume that $(V_1)'$ holds, $0\in\Omega$ and $3p-10-4C_q\|W\|_q>0$. Let $0<a<a_0$ and $r_a>0$ be given in Lemma \ref{lemma-MP-geom}. Then for almost every $s\in[\frac{1}{2},1]$ and $r\geq r_a$, Problem (\ref{eq-Omegar-s}) admits a solution $(\lambda_{r,s},u_{r,s})$ with $u_{r,s}\ge 0$ and $E_{r,s}(u_{r,s})=m_{r,s}(a)$ and
\begin{equation}\label{est-nabla-u-rs'}
\int_{\Omega_r}|\nabla u_{r,s}|^2dx\le\frac{6(p-2)}{3p-10-4C_q^2\|W\|_q} m_{r,s}(a)\le \tilde{c}_a.
\end{equation}
\end{proposition}
\begin{proof}
The proof is identical to the one of Lemma \ref{lemma-unif-bound-s} apart from the estimate
$$\left|\int_{\Omega_r}\nabla V(x)\cdotp x u_{r,s}^2 dx\right|\le \|W\|_q\|u_{r,s}\|_{\frac{2q}{q-1}}^2\le C_q^2\|W\|_q\int_{\Omega_r}|\nabla u_r|^2 dx.$$
\end{proof}

\color{black}
Similarly, Proposition \ref{prop-existence-Omega-r} is replaced by the following result. First we note that there exists $\kappa=\kappa(p,q)>0$ such that , if $\|W\|_q<\kappa$, then $3p-10-4C_q\|W\|_q>0$.
\begin{proposition}\label{prop-existence'}
  Assume that $(V_1)'$ holds, $0\in\Omega$ and $3p-10-4C_q\|W\|_q>0$. Then for any $a\in(0,a_0)$ and $r\ge r_a$ there exists a solution $(\lambda_r,u_r)\in\R\times H^1_0(\Omega_r)$ to Problem (\ref{eq-Omega_r}) such that $E_{r,1}(u_r)=m_{r,1}(a)$, $u_r\ge 0$ fulfilling
				\begin{equation}\label{est-nabla-u-rs'}
                \int_{\Omega_r}|\nabla u_r|^2dx\le\frac{6(p-2)}{3p-10-4C_q^2\|W\|_q}(1+\theta)c_a,\qquad\forall\,r\ge \tilde{r}_a
                \end{equation}
				for some $\tilde{r}_a>r_a$ ($r_a$ is given in Lemma \ref{lemma-MP-geom}), where $\theta:=\vartheta(\|V\|_q)$.  
\end{proposition}
\begin{remark}
Due to Proposition \ref{prop-existence'}, Theorem \ref{th-bd-V_1} still holds true if the assumption $(V_1)$ is replaced by $(V_1)'$ as well. Note that we need no smallness conditions about $\|V\|_q$ and $\|W\|_q$.
\end{remark}
In order to conclude the proof of Proposition \ref{prop-existence-sol} in case $(V_1)'$ holds we observe that, since $\vartheta$ is continuous and $\vartheta(0)=0$, decreasing, if necessary, the value of $\kappa>0$ and taking $V$ such that $$\max\{\|V\|_q,\|W\|_q\}<\kappa(p,q),$$ 
we have $$3p-10-4C_q\|W\|_q>0$$ 
and
$$(p-2)(\|V\|_q +\|W\|_q)C_q^2\frac{6(p-2)(1+\vartheta(\|V\|_q))}{3p-10-4C_q^2\|W\|_q}<6-p,$$
therefore for $\delta>0$ small enough there exists $a_\delta>0$ such that, for any $0<a<a_\delta$, we have
\begin{equation}\notag
					\begin{aligned}
						2(p-2)(C_r+|D_r|)+4(p-3)B_r&\le 2(p-2)(\|V\|_q +\|W\|_q )C_q^2A_r+\tilde{C}a^3 c_a^{1/2}\\
                        &\le 2(p-2)(\|V\|_q +\|W\|_q)C_q^2\frac{6(p-2)(1+\theta)}{3p-10-4C_q^2\|W\|_q}c_a+\tilde{C}a^3 c_a^{1/2}\\
						&<2(6-p-\delta)c_a.
					\end{aligned}
				\end{equation}
		\begin{remark}
	As a consequence Theorem \ref{th-unif-bound} is still true if the assumption $(V_1)$ is replaced by $(V_1)'$ and $\max\{\|V\|_q,\|W\|_q\}<\kappa(p,q)$, where $\kappa(p,q)>0$ is the constant that we found above.     
		\end{remark}

            Proposition \ref{prop-L-infty} is replaced by the following result, which is weaker but still sufficient.
            \begin{proposition}\label{prop-L-infty'}
				Assume that $(V_1)'$ holds, $\Omega$ is a convex bounded Lipschitz domain and the hypothesis of Theorem \ref{th-bd-V_1} are satisfied. Then the solutions $u_r$ constructed in Theorem \ref{th-bd-V_1}  are in $W^{2,2}(\Omega_r)$.
			\end{proposition}
           \begin{proof}
	Due to Proposition \ref{prop-L^q} and the fact that $V\in L^q(\R^3)$ and $\phi_{u_r}\in L^\infty(\R^3)$, we can see that $\Delta u_r\in L^q(\Omega_r)$. Therefore, since $\Omega$ is Lipschitz and $u_r\in H^1_0(\Omega_r)$, the elliptic estimates give that $u_r\in W^{2,q}(\Omega_r)\subset W^{2,2}(\Omega_r)$. 
			\end{proof}
        We note that Proposition \ref{prop-L-infty'} is still sufficient to conclude that $\nabla u_r\cdotp y\in H^1(\Omega_r)$ for any $y\in\R^3$ and $r$ sufficiently large.
			\section{Proof of Theorem \ref{thm1.3}: non-positive potential}\label{sec3}
				\setcounter{equation}{0}
			\setcounter{Assumption}{0} \setcounter{Theorem}{0}
			\setcounter{Proposition}{0} \setcounter{Corollary}{0}
			\setcounter{Lemma}{0}\setcounter{Remark}{0}
			\par
			In this section, we are focused on the non-positive potential case and show that the functional $J_V(u)$ possesses a mountain pass geometry under conditions $(V_3)$ and $(V_4)$, which leads to the existence of a $(PS)$ sequence. Specifically, equation (\ref{p}) admits only normalized solutions with positive energy. 

			\vskip 0.1in
			
			Recall that
			\begin{equation*}
				u^t(x):=t^{\frac{3}{2}}u(tx)~\text{for}~t>0,\,u\in \mathcal{S}_a.
			\end{equation*}
			Then one has
			\begin{equation*}
				J_V(u^t(x))=\frac{t^2}{2}\int_{\mathbb{R}^3}|\nabla u|^2dx+\frac{1}{2}\int_{\mathbb{R}^3}V(\frac{x}{t})u^2dx+\frac{t}{4}B(u)
				-\frac{t^{\frac{3(p-2)}{2}}}{p}\int_{\mathbb{R}^3}|u|^pdx.
			\end{equation*}
			For fixed $u\in \mathcal{S}_a$, it holds
			\begin{equation*}
				\left|\int_{\mathbb{R}^3}V(\frac{x}{t})u^2dx\right|\leq t^2\|V\|_{\frac{3}{2}}\|u\|_6^2\rightarrow0,~\text{as}~t\rightarrow0,
			\end{equation*}
			since $p\in(\frac{10}{3},6)$. from $(V_3)$ one gets
			\begin{equation}\label{eq3.1}
				\lim_{t\rightarrow0^+}J_V(u^t(x))=0,~\lim_{t\rightarrow+\infty}J_V(u^t(x))=-\infty.
			\end{equation}
			In particular, $J_V$ is unbounded from below on $\mathcal{S}_a$. 
			We will show that $J_V$ admits the mountain pass geometry.
			
			\vskip 0.15in
			For any $a\in(0,a_0)$, where $a_0$ is given by Theorem \ref{th1} and $u\in \mathcal{S}_a$, by the Gagliardo-Nirenberg inequality, the Sobolev inequality and the first inequality in $(V_4)$, we have
			\begin{equation}\label{eq3.2}
				\begin{aligned}
					J_V(u )&=\frac{1}{2}\int_{\mathbb{R}^3}|\nabla u|^2dx+\frac{1}{2}\int_{\mathbb{R}^3}V(x)u^2dx+\frac{1}{4}B(u)
					-\frac{1}{p}\int_{\mathbb{R}^3}|u|^pdx\\
					&\geq  \frac{1}{2}\int_{\mathbb{R}^3}|\nabla u|^2dx+\frac{1}{4}B(u)-\frac{1}{p}\big(C(p)\big)^pa^{\frac{6-p}{2}}\|\nabla u\|_2^{\frac{3(p-2)}{2}}-S^{-1}\|V\|_{\frac{3}{2}}\int_{\mathbb{R}^3}|\nabla u|^2dx\\&
					=(\frac{1}{2}-S^{-1}\|V\|_{\frac{3}{2}})\int_{\mathbb{R}^3}|\nabla u|^2dx+\frac{1}{4}B(u)-\frac{1}{p}\big(C(p)\big)^pa^{\frac{6-p}{2}}\|\nabla u\|_2^{\frac{3(p-2)}{2}}\\
					&\geq (\frac{1}{2}-S^{-1}\|V\|_{\frac{3}{2}})\int_{\mathbb{R}^3}|\nabla u|^2dx-\frac{1}{p}\big(C(p)\big)^p  a^{\frac{6-p}{2}}  \|\nabla u\|_2^{\frac{3(p-2)}{2}}.
				\end{aligned}
			\end{equation}
			As a consequence, we can choose $\tilde{R}_1>0$ and $\delta>0$ small such that
			\begin{equation}
				\label{eqK-a}
				K_a:=\inf_{u\in \mathcal{S}_a}\{J_V(u),~\|\nabla u\|_2=\tilde{R}_1\}>\delta>0
			\end{equation}
			with $\delta>0$ independent of $a$, and $J_V(u)>0$ for any $u\in \mathcal{S}_a$ with $\|\nabla u\|_2\le \tilde{R}_1$.
			
			Note that for $u\in \mathcal{S}_a$, by \eqref{local-term}, we have
			\begin{equation}
				\label{supJ-partialBrho}
				J_V(u)\le\frac{1}{2}\|\nabla u\|_2^2+\frac{1}{4}B(u)\le \frac{1}{2}\|\nabla u\|_2^2+Ca^3\|\nabla u\|_2<\frac{\delta}{2}
			\end{equation}
			if $||\nabla u||_2<\tilde{R}_2$, provided $\tilde{R}_2\in(0,\tilde{R}_1)$ is small enough.\\
			
			It follows from (\ref{eqK-a}) and (\ref{supJ-partialBrho}) that $J_V$ satisfies the mountain pass geometry, namely
			$$0<\sup_{u\in\mathcal{S}_a,\|\nabla u\|_2<{\tilde{R}_2}}J_V(u)\le\frac{\delta}{2}<\delta<\inf_{u\in\mathcal{S}_a,\|\nabla u\|_2=\tilde{R}_1} J_V(u)=K_a.$$
			Now we define the mountain pass value 
			\begin{equation*}
				m_{V,a}:=\inf_{\gamma\in \Gamma}\max_{t\in[0,1]}J_V(\gamma(t))\geq K_a,
			\end{equation*}
			where
			\begin{equation*}
				\Gamma=\{\gamma:[0,1]\rightarrow \mathcal{S}_a:\,\|\nabla\gamma(0)\|_2<\tilde{R}_2,~J_V(\gamma(1))<0\}.
			\end{equation*}
			Set
			\begin{equation}
				\mathcal{A}^{\tilde{R}_1}:=\{u\in \mathcal{S}_a:\|\nabla u\|_2>\tilde{R}_1,J_V(u)<0\}
			\end{equation}
			and
			\begin{equation}
				\mathcal{A}_{\tilde{R}_2}:=\{u\in \mathcal{S}_a:\|\nabla u\|_2<\tilde{R}_2,J_V(u)>0\}.
			\end{equation}
			
			Taking a scaling of the solution $v_a\in \mathcal{S}_a$ of (\ref{eq1.7}) given in Theorem \ref{th1}, it is possible to see that there exist $0<s_1<1<s_2$ such that
			\begin{equation*}
				J_V((v_a)^{s_1})<K_a,~\text{if}~\|\nabla (v_a)^{s_1}\|_2<\tilde{R}_2,
			\end{equation*}
			\begin{equation*}
				J_V((v_a)^{s_2})<0,~\text{if}~\|\nabla (v_a)^{s_2}\|_2>\tilde{R}_1.
			\end{equation*}
			Finally, the path $g(t)=(v_a)^{(1-t)s_1+ts_2}$ ($t\in[0,1]$) can be used to show that
			\begin{equation}\label{comparison}
				m_{V,a}<c_a=\inf_{g\in \mathcal{G}_a}\max_{t\in[0,1]}I(g(t)),
			\end{equation}
			which will be useful in the sequel. \\
			
			To construct a  $(PS)$ sequence of $J_V$, we recall the following result, which is a special case of \cite[Theorem 4.5]{G1993}.
			\begin{proposition}\label{pro4.1}
				Let $M$ be a Hilbert manifold, $\mathcal{I}\in C^1(M,\mathbb{R})$ be a given functional and $K\subset M$ be compact. Suppose that the subset
				\begin{equation*}
					\mathcal{C}\subset \{C\subset M:C ~\text{is~compact},K\subset C\}
				\end{equation*}
				is homotopy--stable, i.e., it is invariant with respect to deformations leaving $K$ fixed. Moreover, let
				\begin{equation*}
					\max_{u\in K}\mathcal{I}(u)<c:=\inf_{C\subset  \mathcal{C}}\max_{u\in C} \mathcal{I}(u)\in \mathbb{R},
				\end{equation*}
				let $\{\sigma_n\}\subset\mathbb{R} $ be a sequence such that $\sigma_n\rightarrow0$ and $\{C_n\}\subset \mathcal{C}$ be a sequence such that
				\begin{equation*}
					0\leq \max_{u\in C_n}\mathcal{I}(u)-c\leq \sigma_n.
				\end{equation*}
				Then there exists a sequence $\{v_n\}\subset M$ satisfying
				\begin{description}
					\item[(1)] $|\mathcal{I}(v_n)-c|\leq \sigma_n$,
					\item[(2)] $||\nabla_M\mathcal{I}(v_n)||\leq C_1\sqrt{\sigma_n}$,
					\item[(3)] dist$(v_n,C_n)\leq C_2\sqrt{\sigma_n}$.
				\end{description}
			\end{proposition}
			Apply Proposition \ref{pro4.1}, we can construct a $(PS)$ sequence of $J_V$ at the level $m_{V,a}$.
			\begin{lemma}\label{lem2.7}
				Assume that $(V_3)-(V_4)$ hold. Then there exists a $(PS)$ sequence $\{u_n\}\subset \mathcal{S}_a$ such that as $n\rightarrow\infty,$
				\begin{equation}\label{eq2.14}
					J_V(u_n)\rightarrow m_{V,a},~\nabla _{\mathcal{S}_a}J_V(u_n)\rightarrow0
				\end{equation}
				and
				\begin{equation}\label{e2.16}
					\|\nabla u_n\|_2^2+\frac{1}{4}B(u_n)-\frac{3(p-2)}{2p}\|u_n\|_p^p+\frac{1}{2}\int_{\mathbb{R}^3}V(x)(3|u_n|^2+2u_n\nabla u_n\cdot x)dx\rightarrow0.
				\end{equation}
			\end{lemma}
			\begin{proof}
				The proof closely follows the arguments in \cite[Proposition 3.11]{bm}. For completeness, we will give the main strategy. 
				
				Choose a sequence $\{\gamma_n\}\subset \Gamma$ such that
				\begin{equation*}
					\max_{t\in[0,1]}J_V(\gamma_n(t))\leq m_{V,a}+\frac{1}{n}.
				\end{equation*}
				Without loss of generality, we can assume that $\gamma_n(t)\geq0$ a.e. in $\mathbb{R}^3$ due to the fact that $J_V(u)=J_V(|u|)$ for any $u\in H^1(\mathbb{R}^3)$.  Define
				\begin{equation*}
					\tilde{\Gamma}_a=\{\tilde{\gamma}:[0,1]\rightarrow \mathcal{S}_a\times \R:\, \tilde{\gamma}(0)\in \mathcal{A}_{\tilde{R}_2}\times\{1\},~ \tilde{\gamma}(1)\in\mathcal{A}^{\tilde{R}_1}\times\{1\}\}.
				\end{equation*}
				
				To apply Proposition \ref{pro4.1}, we set $\mathcal{I}(u)=\tilde{J}_V(u,t):=J_V(u^t)$ and
				\begin{equation*}
					M:=\mathcal{S}_a\times\R,~K:=\{\tilde{\gamma}(0),\tilde{\gamma}(1)\},~\mathcal{C}:=\tilde{\Gamma}_a,~C_n:=\{(\gamma_n(t),1):t\in[0,1]\}.
				\end{equation*}
				Note that $\tilde{\gamma}_n:=(\gamma_n(t),1)\in \tilde{\Gamma}_a$. By Proposition \ref{pro4.1}, there exists $(v_n,t_n)\in H^1(\mathbb{R}^3)\times \mathbb{R}$ such that  as $n\rightarrow\infty$,
				\begin{equation*}
					\tilde{J}_V(v_n,t_n)\rightarrow m_{V,a},~D\tilde{J}_V(v_n,t_n)\rightarrow0.
				\end{equation*}
				Moreover, we get
				\begin{equation*}
					\min_{t\in[0,1]}\|(v_n,t_n)-(\gamma_n(t),1)\|_{H^1(\mathbb{R}^3)\times \mathbb{R}}\leq \frac{\tilde{C}}{\sqrt{n}},
				\end{equation*}
				thus $t_n\rightarrow 1$ and there exists $s_n\in [0,1]$ with $\|v_n-\gamma_n(s_n)\|\rightarrow0$ as $n\rightarrow\infty$.
				
				Define $u_n:=(v_n)^{t_n}.$ Since $\gamma_n(t)\geq0$ a.e. in $\mathbb{R}^3$, then $\|v_n^-\|_2\leq \|v_n-\gamma_n(s_n)\|_2=o(1)$, so that $v_n^-\rightarrow0$ a.e. in $\mathbb{R}^3$, up to a subsequence. So
				\begin{equation}\label{eq2.15}
					\|u_n^-\|_2\rightarrow0~\text{as}~n\rightarrow\infty,
				\end{equation}
				Now we show $\{u_n\}$ is a $(PS)$ sequence for $J_V$. It is easy to see 
				$J_V(u_n)\rightarrow m_{V,a}$  as $n\rightarrow\infty$. 
				For each $w\in H^1(\R^3),$ set $w_n:=(w)^{-t_n}$, then one can deduce that
				\begin{equation*}
					\nabla (J_V-J_{\infty})(u_n)[w]=\int_{\mathbb{R}^3}V(\frac{x}{t_n})v_nw_ndx,
				\end{equation*}
				which means
				\begin{equation*}
					DJ_V(u_n)[w]=D\tilde{J}_V(v_n,t_n)[(w_n,1)]+o(1)||w_n||.
				\end{equation*}
				Moreover, $\int_{\mathbb{R}^3}v_nwdx=0$ is equivalent to $\int_{\mathbb{R}^3}u_nw_ndx=0$.
				By the definition of $w_n$, we know that for $n$ large, $||w_n||^2_{H^1(\mathbb{R}^3)}\leq 2||w||^2_{H^1(\mathbb{R}^3)}$, thus (\ref{eq2.14}) holds.
				
				Note that $\text{as}~n\rightarrow\infty$,
				\begin{equation*}
					D\tilde{J}_V(v_n,t_n)[(0,1)]\rightarrow0,~t_n\rightarrow1.
				\end{equation*}
				An explicit computation shows that
				\begin{equation*}
					\partial_t\big(\int_{\mathbb{R}^3}V(x)t^{3}u^2(tx)dx\big)=\int_{\mathbb{R}^3}V(x)(3t^3u^2(tx)+2t^3u(tx)\nabla u(tx)\cdot tx)dx
				\end{equation*}
				and
				\begin{equation*}
					\partial_t I(u^t)=\|\nabla u\|_2^2+\frac{1}{4}B(u)-\frac{3(p-2)}{2p}\|u\|_p^p,
				\end{equation*}
				then we have $\text{as}~n\rightarrow\infty$,
				\begin{equation*}
					\|\nabla u_n\|_2^2+\frac{1}{4}B(u_n)-\frac{3(p-2)}{2p}\|u_n\|_p^p+\frac{1}{2}\int_{\mathbb{R}^3}V(x)(3|u_n|^2+2v_n\nabla u_n\cdot x)dx\rightarrow0,
				\end{equation*}
				which means $u_n$ almost satisfies the Pohozaev identity and \eqref{e2.16} holds. 
				This completes the proof.
			\end{proof}
			
			We will now demonstrate that the $(PS)$ sequence obtained in Lemma \ref{lem2.7} is bounded in $H^1(\mathbb{R}^3)$.
			
			\begin{lemma}\label{lem3.1}
				Assume that $(V_3)-(V_4)$ hold. Let $\{u_n\}\subset \mathcal{S}_a$ be a $(PS)$ sequence for $J_V$ obtained in Lemma \ref{lem2.7}.  Then $\{u_n\}$ is bounded in $H^1(\mathbb{R}^3)$. Moreover, there exists a sequence of Lagrange multipliers with
				\begin{equation*}
					\lambda_n:=-\frac{DJ_V(u_n)[u_n]}{a^2}
				\end{equation*}
				and $\lambda_n\rightarrow \tilde{\lambda}>0$ for any $a\in(0,a_*)$, where
				\begin{equation*}
					a_*=\Big(\frac{6-p}{|6-2p|\hat{C}}\Big)^{\frac{1}{3}}\Big(\frac{\delta}{\tilde{\eta}}\Big)^{\frac{1}{6}}
				\end{equation*}
				and $\hat{C},\delta,\hat{\eta}$ are positive constants independent of $a$.
			\end{lemma}
			\begin{proof}
				It follows from Lemma \ref{lem2.7} that $\{u_n\}$ almost satisfies the following Pohozaev identity
				\begin{equation}\label{eq3.3}
					\|\nabla u_n\|_2^2+\frac{1}{4}B(u_n)-\frac{3(p-2)}{2p}\|u_n\|_p^p+\frac{1}{2}\int_{\mathbb{R}^3}V(x)(3|u_n|^2+2u_n\nabla u_n\cdot x)dx\rightarrow0.
				\end{equation}
				Setting
				\begin{equation}\label{eq3.4}
					\begin{aligned}
						&\tilde{A}_n:=\|\nabla u_n\|_2^2, ~~~\tilde{B}_n:=B(u_n)=\int_{\mathbb{R}^3}\int_{\mathbb{R}^3}\frac{|u_n(x)|^2|u_n(y)|^2}{|x-y|}dxdy,\\
						&~\tilde{C}_n:=
						-\int_{\mathbb{R}^3}V(x)u_n^2dx,~\tilde{D}_n:=-\int_{\mathbb{R}^3}V(x)u_n\nabla u_n\cdot xdx,~\tilde{E}_n:=\int_{\mathbb{R}^3}|u_n|^pdx,
					\end{aligned}
				\end{equation}
				we have
				\begin{equation}\label{eq3.5}
					\tilde{A}_n-\tilde{C}_n+\frac{1}{2}\tilde{B}_n-\frac{2}{p}\tilde{E}_n=2m_{V,a}+o(1),~\text{as}~n\rightarrow\infty,
				\end{equation}
				\begin{equation}\label{eq3.6}
					\tilde{A}_n+\frac{1}{4}\tilde{B}_n-\frac{3(p-2)}{2p}\tilde{E}_n-\frac{3}{2}\tilde{C}_n-\tilde{D}_n=o(1),~\text{as}~n\rightarrow\infty,
				\end{equation}
				\begin{equation}\label{eq3.7}
					\tilde{A}_n-\tilde{C}_n+\lambda_na^2+\tilde{B}_n=\tilde{E}_n+o(1)(a_n^{1/2}+1),~\text{as}~n\rightarrow\infty.
				\end{equation}
				By (\ref{eq3.5}) and (\ref{eq3.6}), we have
				\begin{equation}\label{eq3.8}
					\frac{3p-8}{p}\tilde{E}_n=2m_{V,a}+\tilde{A}_n-2\tilde{C}_n-2\tilde{D}_n+o(1).
				\end{equation}
				Using (\ref{eq3.5}) again, one gets
				\begin{equation*}
					\frac{3p-10}{3p-8}\tilde{A}_n+\frac{1}{2}\tilde{B}_n= \frac{3p-12}{3p-8}\tilde{C}_n- \frac{4}{3p-8}\tilde{D}_n+ \frac{6(p-2)}{3p-8}m_{V,a}+o(1).
				\end{equation*}
				Since $|\tilde{C}_n|\leq S^{-1}\|V\|_{\frac{3}{2}}\tilde{A}_n$, ~$|\tilde{D}_n|\leq S^{-\frac{1}{2}}\|\tilde{W}\|_3\tilde{A}_n, $ it follows that
				\begin{equation*}
					\frac{3p-10}{3p-8}\tilde{A}_n\leq \frac{3(p-4)^+}{3p-8}S^{-1}\|V\|_{\frac{3}{2}}\tilde{A}_n+\frac{4}{3p-8}S^{-\frac{1}{2}}\|\tilde{W}\|_3\tilde{A}_n+ \frac{6(p-2)}{3p-8}m_{V,a},
				\end{equation*}
				where $(p-4)^+=\max\{p-4,0\}$,
				which leads to $\tilde{A}_n\leq\tilde{ \eta}m_{V,a}<\tilde{ \eta} m_a$, here we have used the second inequality of $(V_4)$ and the fact
				\begin{equation*}
					\tilde{ \eta}:=\frac{ 6(p-2)}{3p-10- 3(p-4)^+S^{-1}\|V\|_{\frac{3}{2}}-4S^{-\frac{1}{2}}\|\tilde{W}\|_3}>0.
				\end{equation*}
				Therefore, $\|\nabla u_n\|_2$ is bounded in $\R$. As a consequence we get that $\tilde{C}_n,\tilde{D}_n$ are both bounded, then by H\"{o}lder inequality, one gets  $\tilde{B}_n,\tilde{E}_n,\lambda_n$ are also bounded. Along a subsequence, there hold
				\begin{equation*}
					\tilde{A}_n\rightarrow \tilde{A}\geq0,~\tilde{B}_n\rightarrow\tilde{ B}\geq0, ~\tilde{C}_n\rightarrow \tilde{C}\geq0, ~\tilde{D}_n\rightarrow \tilde{D}\in \mathbb{R},\lambda_n\rightarrow\tilde{\lambda}\in \mathbb{R}.
				\end{equation*}
				It follows from $(V_4)$ that
				\begin{equation}
					\begin{aligned}
						\tilde{\lambda}a^2&=\frac{p-2}{p}E-\frac{1}{2}B-2m_{V,a}\\&
						=\frac{2(p-2)}{3(p-2)-4}\big(2m_{V,a}-\frac{1}{2}\tilde{C}-\tilde{D}-\frac{1}{4}\tilde{B}\big)-\frac{1}{2}\tilde{B}-2m_{V,a}\\&
						=\frac{6-p}{3(p-2)-4}2m_{V,a}-\frac{p-2}{3(p-2)-4}\tilde{C}-\frac{2(p-2)}{3(p-2)-4}\tilde{D}
						+\frac{6-2p}{3(p-2)-4}\tilde{B}\\&
						\geq\frac{6-p}{3(p-2)-4}2m_{V,a}-\frac{p-2}{3(p-2)-4}S^{-1}\|V\|_{\frac{3}{2}}\tilde{A}_n-
						\frac{2(p-2)}{3(p-2)-4}S^{-\frac{1}{2}}\|\tilde{W}\|_3\tilde{A}_n\\
						&~~~+\frac{6-2p}{3(p-2)-4}\tilde{B}\\&
						\geq \frac{6-p}{3(p-2)-4}2m_{V,a}-\frac{p-2}{3(p-2)-4}S^{-1}\|V\|_{\frac{3}{2}}\tilde{ \eta}m_{V,a}-
						\frac{2(p-2)}{3(p-2)-4}S^{-\frac{1}{2}}\|\tilde{W}\|_3\tilde{ \eta}m_{V,a}\\
						&~~~+\frac{6-2p}{3(p-2)-4}\tilde{B}.
					\end{aligned}
				\end{equation}
				By the Gagliardo-Nirenberg inequality,  there exists constant $\hat{C}$ independent of $a$ such that
				\begin{equation*}
					\begin{aligned}
						\tilde{B}&\leq \limsup_{n\to\infty}\int_{\mathbb{R}^3}\phi_{u_n}u_n^2 dx\leq \hat{C}\limsup_{n\to\infty}\|u_n\|_{\frac{12}{5}}^4\\&\leq \hat{C}\limsup_{n\to\infty} \|\nabla u_n\|_2\|u_n\|_2^{3}\le \hat{C} \tilde{A}^{\frac{1}{2}}a^{3}\leq \hat{C}(\tilde{ \eta}m_{V,a})^{\frac{1}{2}}a^3,
					\end{aligned}
				\end{equation*}
				so that $\tilde{\lambda}>0$ provided
				\begin{equation}
					\label{eq-lambda>0}
					2(6-p)-(p-2)S^{-1}\|V\|_{\frac{3}{2}}\tilde{ \eta}-
					2(p-2)S^{-\frac{1}{2}}\|\tilde{W}\|_3\tilde{ \eta}+(6-2p)\hat{C}\tilde{\eta}^{\frac{1}{2}}\frac{a^3}{(m_{V,a})^{\frac{1}{2}}}>0.
				\end{equation}
				By the definition, we know $m_{V,a}>\delta>0$ with $\delta$ independent of $a$, thus
				\begin{equation*}
					|6-2p|\hat{C}\tilde{\eta}^{\frac{1}{2}}\frac{a^3}{(m_{V,a})^{\frac{1}{2}}}< |6-2p|\hat{C}\tilde{\eta}^{\frac{1}{2}}\frac{a^3}{\delta^{\frac{1}{2}}}.
				\end{equation*}
				Let $a_*=\Big(\frac{6-p}{|6-2p|\hat{C}}\Big)^{\frac{1}{3}}\Big(\frac{\delta}{\tilde{\eta}}\Big)^{\frac{1}{6}}$. Then one can check for any $a\in(0,a_*)$ 
				\begin{equation*}
					|6-2p|\hat{C}\tilde{\eta}^{\frac{1}{2}}\frac{a^3}{(m_{V,a})^{\frac{1}{2}}}<6-p.
				\end{equation*}
				By combining this with the second inequality in condition $(V_4)$, we can confirm (\ref{eq-lambda>0}) holds.
				
				This completes the proof.
			\end{proof}
			\begin{lemma}\label{lemma4.2}
				Assume that $(V_3)-(V_4)$ hold. If $v\in \mathcal{S}_a$ is a normalized solution of (\ref{p}) for  some $a>0$, then $J_V(v)\geq0$.
			\end{lemma}
			\begin{proof}
				Since $v$ is a solution of (\ref{p}), then $v$ satisfies the Pohozaev identity \eqref{e2.16}, i.e.,
				\begin{equation}\label{eq3.9}
					\begin{aligned}
						\frac{1}{p}\|v\|_p^p=& \frac{2}{3(p-2)}\|\nabla v\|_2^2+\frac{1}{6(p-2)}B(v)+\frac{1}{p-2}\int_{\mathbb{R}^3}V(x)|v|^2dx\\&
						+\frac{2}{3(p-2)}\int_{\mathbb{R}^3}V(x)v\nabla v\cdot xdx.
					\end{aligned}
				\end{equation}
				From where one gets
				\begin{equation}\label{eq3.10}
					\begin{aligned}
						J_V(v )&=\frac{1}{2}\int_{\mathbb{R}^3}|\nabla v|^2dx+\frac{1}{2}\int_{\mathbb{R}^3}V(x)v^2dx+\frac{1}{4}B(v)
						-\frac{1}{p}\int_{\mathbb{R}^3}|v|^pdx\\
						&= \big(\frac{1}{2}-\frac{2}{3(p-2)}\big)\|\nabla v\|_2^2+\big(\frac{1}{4}-\frac{1}{6(p-2)}\big)B(v)+\big(\frac{1}{2}-\frac{1}{p-2}\big)\int_{\mathbb{R}^3}V(x)|v|^2dx
						\\&~~~-\frac{2}{
3(p-2)}\int_{\mathbb{R}^3}V(x)v\nabla v\cdot xdx\\&
						=\frac{3p-10}{6(p-2)}\|\nabla v\|_2^2+\frac{3p-8}{12(p-2)}B(v)+\frac{p-4}{2(p-2)}\int_{\mathbb{R}^3}V(x)|v|^2dx \\&~~~-\frac{2}{3(p-2)}\int_{\mathbb{R}^3}V(x)v\nabla v\cdot xdx.
					\end{aligned}
				\end{equation}
				Since $V(x)\leq0$ and $p\in(\frac{10}{3},6)$, then using the second inequality in $(V_4)$ again, we get
				\begin{equation*}
					\begin{aligned}
						J_V(v )&\geq\frac{3p-10}{6(p-2)}\|\nabla v\|_2^2-\frac{(p-4)^+}{2(p-2)}S^{-1}\|V\|_{\frac{3}{2}}\|\nabla v\|_2^2-\frac
						{2}{3(p-2)}\int_{\mathbb{R}^3}V(x)v\nabla v\cdot xdx\\&\geq
						\Big( \frac{3p-10}{6(p-2)}-\frac{(p-4)^+}{2(p-2)}S^{-1}\|V\|_{\frac{3}{2}}-\frac{2}{3(p-2)}S^{-\frac{1}{2}}\|\tilde{W}\|_3\Big)\|\nabla u\|_2^2>0.
					\end{aligned}
				\end{equation*}
				This proves the lemma.
			\end{proof}
			\begin{remark}
				Lemma \ref{lemma4.2} reveals the fact that under the conditions $(V_3)-(V_4)$, (\ref{p}) does not admit normalized solutions with negative energy.
			\end{remark}
			\begin{proof}[\textbf{Proof of Theorem \ref{thm1.3}}]  By Lemma  \ref{lem2.7} and  \ref{lem3.1}, we know there is $a_*>0$  such that for any $a\in(0,a_*)$, there exists a bounded $(PS)$ sequence $\{u_n\}\subset \mathcal{S}_a$ for $J_V$ at level $m_{V,a}$ and the associated Lagrange multipliers $\lambda_n\rightarrow\tilde{\lambda}>0$. Then there exists $\tilde{u}\in H^1(\R^3)$ such that, along a subsequence, $  u_n\rightharpoonup \tilde{u} ~\text{in}~H^1(\mathbb{R}^3)$ as $n\rightarrow\infty$.

			\vskip0.1in
			To finish the proof of Theorem \ref{thm1.3}, we need to prove as $n\rightarrow\infty$,
			\begin{equation*}
				u_n\rightarrow \tilde{u}\quad \text{strongly\quad in}\quad H^1(\R^3).
			\end{equation*}
			Since  $\{u_n\}\subset \mathcal{S}_a$ is a $(PS)$ sequence of $J_V$, then for any $\psi\in H^1(\mathbb{R}^3)$, it holds
			\begin{equation*}
				\begin{aligned}
					&\int_{\mathbb{R}^3}\nabla u_n\nabla \psi dx+\int_{\R^3}V(x)u_n\psi dx+\int_{\mathbb{R}^3}(|x|^{-1}*u_n^2)u_n\psi dx
					\\&=-\lambda_n\int_{\mathbb{R}^3}u_n\psi dx+\int_{\mathbb{R}^3}|u_n|^{p-2}u_n\psi dx+o(1)\|\psi\|
				\end{aligned}
			\end{equation*}
			and using the fact $\lambda_n\rightarrow\tilde{\lambda}$, one gets
			\begin{equation*}
				\begin{aligned}
					&\int_{\mathbb{R}^3}\nabla u_n\nabla \psi dx+\int_{\R^3}V(x)u_n\psi dx+\int_{\mathbb{R}^3}(|x|^{-1}*u_n^2)u_n\psi dx
					\\&=-\tilde{\lambda}\int_{\mathbb{R}^3}u_n\psi dx+\int_{\mathbb{R}^3}|u_n|^{p-2}u_n\psi dx+o(1)\|\psi\|.
				\end{aligned}
			\end{equation*}
			The above results show that $\{u_n\}$ is actually a $(PS)$ sequence for $J_{V,\tilde{\lambda}}$ at level $m_{V,a}+\frac{\tilde{\lambda}}{2}a^2$. Since $(V_4)$ holds, we have $\|V\|_{\frac{3}{2}}=\|V^-\|_{\frac{3}{2}}<S$. Hence we can apply the splitting Lemma \ref{lem2.2} to say that
			\begin{equation*}
				u_n=\tilde{u}+\sum_{j=1}^kw^j(\cdot-y_n^j )+o(1),
			\end{equation*}
			\begin{equation*} 
				\|u_n\|_2^2=\|\tilde{u}\|_2^2+\sum_{j=1}^k\|w^j\|_2^2+o(1)
			\end{equation*}
			and
			\begin{equation*}
				J_{V,\tilde{\lambda}}(u_n)= J_{V,\tilde{\lambda}}(\tilde{u})+\sum_{j=1}^kJ_{\infty,\tilde{\lambda}}(w^j)+o(1),
			\end{equation*}
			where $w^j$ ($1\leq j\leq k$) is the solution
			to the limit equation
			\begin{equation*}
				-\Delta u+\tilde{\lambda} u+(|x|^{-1}*|u|^2)u=|u|^{p-2}u.
			\end{equation*}
			If $k\geq1$, then $\|\tilde{u}\|_2<a$. Thus one has
			\begin{equation*} m_{V,a}+\frac{\tilde{\lambda}}{2}a^2=J_V(u)+\frac{\tilde{\lambda}}{2}\gamma^2+\sum_{j=1}^{k}\left(J_{\infty}(w^j)+\frac{\tilde{\lambda}}{2}\rho_j^2\right),
			\end{equation*}
			where $\gamma:=\|u\|_2$, $\rho_j:=\|w^j\|_2$. Note that $a^2=\gamma^2+\sum_{j=1}^{k}\rho_j^2$, then we have
			\begin{equation}\label{eq3.11}
				m_{V,a} =J_V(u) +\sum_{j=1}^{k} J_{\infty}(w^j) .
			\end{equation}
			Recall that for $0<a_1\leq a_2,$ we have $c_{a_1}\geq c_{a_2}$ (see \cite[Theorem 1.2-(ii)]{BJ2013}), therefore, $J_{\infty}(w^j)\geq c_a$. Gathering  lemma \ref{lemma4.2} with (\ref{eq3.11}), we obtain that $m_{V,a}\geq c_a$ which contradicts the fact $m_{V,a}<c_a$ (see \eqref{comparison}). Thus $k=0$ and $u_n\rightarrow \tilde{u}$ strongly in $H^1(\mathbb{R}^3)$. Since $\{u_n\}$ is non-negative, we know $\tilde{u}\geq0$, which concludes the proof of Theorem \ref{thm1.3}. 
			\end{proof}


            \noindent{\bf Declarations of interest}
			
			The authors have no interest to declare.

			~\\
			\color{black}\noindent{\bf Acknowledgment}
			
			X.Q. Peng is partially supported by the China Scholarship Council (No.202208310170). M. Rizzi is partly supported by Justus Liebig University and the DFG project number $62202684$. The authors thank Prof. Bartsch for his suggestions to improve this work.


				
			\end{document}